\documentclass[a4paper,12pt]{article}
\usepackage{amsmath,amsthm,amsfonts,amssymb,bm,mathrsfs}
\usepackage{enumerate}
\usepackage{graphicx}
\usepackage[numbers]{natbib}
\usepackage{xcolor}

\usepackage{sgame}

\usepackage{setspace}
\newtheorem{defn}{Definition}
\newtheorem{thm}{Theorem}

\newtheorem{lem}{Lemma}

\newtheorem{prop}{Proposition}
\newtheorem{rmk}{Remark}
\newtheorem{exam}{Example}
\newtheorem{coro}{Corollary}

\newtheorem{claim}{Claim}

\newcommand{\bE}{\mathbb{E}}

\newcommand{\bR}{\mathbb{R}}
\newcommand{\cB}{\mathcal{B}}
\newcommand{\cD}{\mathcal{D}}

\newcommand{\cF}{\mathcal{F}}
\newcommand{\cG}{\mathcal{G}}
\newcommand{\cH}{\mathcal{H}}
\newcommand{\cI}{\mathcal{I}}

\newcommand{\cM}{\mathcal{M}}

\newcommand{\cR}{\mathscr{R}}
\newcommand{\cS}{\mathcal{S}}

\DeclareMathOperator{\rmd}{d\!}

\usepackage{hyperref}

\hypersetup{colorlinks=true,
           citecolor=blue,
           linkcolor=blue,
           urlcolor=blue,
           bookmarksopen=true,
           pdfstartview={XYZ null null 1.70},
           pdfpagelayout={SinglePage}
}

\allowdisplaybreaks

\begin{document}

\title{Stationary Markov Perfect Equilibria in Discounted Stochastic Games\footnote{The authors are grateful to Rabah Amir, Darrell Duffie, Matthew Jackson, Jiangtao Li, Xiang Sun, Yifei Sun, Satoru Takahashi, Bin Wu and Yongchao Zhang for helpful discussions. This work was presented in Institute of Economics, Academia Sinica, Taipei, November 2013; 5th Shanghai Microeconomics Workshop, June 2014; Asian Meeting of the Econometric Society, Taipei, June 2014; China Meeting of the Econometric Society, Xiamen, June 2014; 14th SAET Meeting, Tokyo, August 2014; the Conference in honor of Abraham Neyman, Jerusalem, June 2015; Department of Economics, Rutgers University, October 2016; and 2016 NSF/NBER/CEME Mathematical Economics Conference, Johns Hopkins University, October 2016. This research was supported in part by the Singapore Ministry of Education Academic Research Fund Tier 1 grants R-122-000-227-112 and R-146-000-170-112. This version owes substantially to the careful reading and expository suggestions of an editor, an associate editor and two referees.
}}
\author{Wei~He\thanks{Department of Economics, Chinese University of Hong Kong, Hong Kong. E-mail: he.wei2126@gmail.com}
\and
Yeneng~Sun\thanks{Department of Economics, National University of Singapore, 1 Arts Link, Singapore 117570. Email: ynsun@nus.edu.sg}
}

\date{\today}

\maketitle

\abstract{\singlespace{The existence of stationary Markov perfect equilibria in stochastic games is shown under a general condition called ``(decomposable) coarser transition kernels''. This result covers various earlier existence results on correlated equilibria, noisy stochastic games, stochastic games with finite actions and state-independent transitions, and stochastic games with mixtures of constant transition kernels as special cases. A remarkably simple proof is provided via establishing a new connection between stochastic games and conditional expectations of correspondences. New applications of stochastic games are presented as illustrative examples, including stochastic games with endogenous shocks and a stochastic dynamic oligopoly model.}

{\singlespace{\textbf{Keywords}: Stochastic game, stationary Markov perfect equilibrium, (decomposable) coarser transition kernel, endogenous shocks, dynamic oligopoly.}
}}

\newpage

\tableofcontents

\newpage

\section{Introduction}\label{sec-intro}

The class of stochastic games enriches the model of repeated games by allowing the stage games to vary with some publicly observable states. It has considerably widened the applications of repeated games.\footnote{See, for example, the book \cite{NS2003} (in particular the survey chapter \cite{Amir2003} on economic applications), and the recent survey chapters \cite{JN2016a} and \cite{JN2016b}.} In particular, a stochastic game is played in discrete time by a finite set of players and the past history is observable by all the players. At the beginning of each stage, Nature draws some state randomly. After the state is realized, the players choose actions and each player receives a stage payoff that depends on the current state and the actions. The game then moves to the next stage and a new random state is drawn, whose distribution depends on the previous state and chosen actions. The procedure is repeated at the new state. Each player's total payoff is the discounted sum of the stage payoffs.

A strategy for a player in a stochastic game is a complete plan of actions, which specifies a feasible action for the player in every contingency in which the player might be called on to act.  However, the so-called Markov strategies, which only depend on the current state instead of the entire past history of states and action profiles, have received much attention in the literature. As noted in \cite{MT2001}, the concept of Markov perfect equilibrium, which requires the players to use only Markov strategies, embodies various practical virtues and philosophical considerations, including conceptual and computational simplicity. Given that the relevant parameters in a stochastic game are time-independent, it is natural to require the Markov strategies to be time-independent as well. Equilibria based on such strategies are called stationary Markov perfect equilibria. In a stationary Markov perfect equilibrium, any subgames with the same current states
will be played exactly in the same way. So ``bygones'' are really ``bygones''; i.e., the past history does not matter at all.

Beginning with \cite{Shapley1953}, the existence of stationary Markov perfect equilibria in discounted stochastic games remains an important problem. Existence results on such equilibria in stochastic games with compact action spaces and finite/countable state spaces have been established in a number of early papers.\footnote{The result for zero-sum games with finite actions and states was established in the seminal paper by \cite{Shapley1953}.
The case with finite state spaces and compact action spaces was shown in \cite{Fink1964} and \cite{Takahashi1964}. The result in \cite{Shapley1953} was extended by \cite{Parthasarathy1973} to nonzero-sum games with finite action spaces and  countable state spaces. The existence result for countable state spaces and compact action spaces was then proved by \cite{Federgruen1978}. Approximate equilibrium has been considered in \cite{Nowak1985} and \cite{Whitt1980}. For more detailed discussions about the literature on stochastic games, see the surveys \cite{JN2016a, JN2016b}. \label{fn-literature}} Given the wide applications of stochastic games with general state spaces in various areas of economics, the existence of equilibria in stationary strategies for such games has been extensively studied in the last two decades. However, no general existence result, except for several special classes of stochastic games, has been obtained in the literature so far.\footnote{Stochastic games possessing strategic complementarities were studied in \cite{Amir2002}, \cite{BRW2014}, \cite{Curtat1996}, \cite{JN2015}, \cite{Nowak2007} and \cite{Vives2009}; see Section 6 in the survey  \cite{Vives2005} for further discussions. A related class of stochastic games in which the interaction between different players is sufficiently weak was studied in \cite{Horst2005}. The existence of stationary $p$-equilibria for two-player games with  finite actions and strong separability conditions on both stage payoffs and state transitions was proved in \cite{HPRV1976}.  The existence of stationary Markov perfect equilibria in intergenerational stochastic games was considered in \cite{JN2014} and \cite{Nowak2006}.
Stochastic games with sub-modularity were studied in \cite{Amir1996} and \cite{BN2008}. All these papers impose special conditions on the payoff functions and state transitions. }

The existence of stationary Markov perfect correlated equilibria for stochastic games was proved in \cite{DGMM1994} and \cite{NR1992}.\footnote{Additional ergodic properties were obtained in \cite{DGMM1994} under stronger conditions.} These papers assumed that there is a randomization device publicly known to all players which is irrelevant to the fundamental parameters of the game. Stationary Markov perfect equilibria have been shown to exist in \cite{Duggan2012}, \cite{Nowak2003} and \cite{PS1989} for stochastic games with some special structures on the state transitions. The paper \cite{PS1989} focused on stochastic games with finite actions and state-independent transitions. A class of stochastic games has been studied in \cite{Nowak2003}, where the transition probability has a fixed countable set of atoms and its atomless part is a finite combination of atomless measures that do not depend on states and actions.  Stochastic games with a specific product structure, namely stochastic games with noise, were considered in \cite{Duggan2012}. In particular, the noise variable in \cite{Duggan2012} is a history-irrelevant component of the state and could influence the payoff functions and transitions. Recently, it was shown in \cite{Levy2013} and \cite{LM2015} that a stochastic game satisfying the usual conditions as stated in Section \ref{sec-game} below may not have any stationary Markov perfect equilibrium.\footnote{An error in the relevant example of \cite{Levy2013} was pointed out in \cite{LM2015}, and a new example was presented therein. In the new example, the discounted stochastic game has a continuum of states, finitely many players and actions, where all the state transitions are absolutely continuous with respect to a fixed measure. Since there are only finitely many actions in this stochastic game, any continuity assumptions in terms of actions will automatically be satisfied.} This implies that a general existence result could only hold under some suitable conditions.

The first contribution of this paper is to introduce a general condition to guarantee the existence of stationary Markov perfect equilibria in stochastic games. Based on this condition, we unify various existence results in the literature as discussed in the previous paragraph,\footnote{All those papers work with general payoffs, but impose special structure on the state transitions.} and also provide a new class of stochastic games that is useful for economic applications and cannot be handled by existing  results. Our second contribution is methodological. We establish for the first time a connection between the equilibrium payoff correspondences in stochastic games and a general result on the conditional expectations of correspondences, and hence are able to provide a simple proof for the existence of stationary Markov perfect equilibria. In the following paragraphs, we will discuss our condition and results in details.

The transition probability in a general stochastic game is defined in terms of the actions and state in the previous period. As discussed above, it is assumed in \cite{DGMM1994}, \cite{Duggan2012}, \cite{NR1992} and \cite{PS1989} that the actions and state in the previous period do not enter the transition of the sunspot/noise/shock component of the states.
As a result, a key component of the transition probability in those papers is not influenced by the actions and state in the previous period.
However, dynamic economic models with random shocks are common, and it may seem natural to assume that the transition of shocks can endogenously depend on some important factors (actions/states) in the previous stage. As shown by the counterexample in \cite{LM2015}, without any restriction on the transition kernel,\footnote{Here the transition kernel means the Radon-Nikodym derivative of the transition probability with respect to some reference measure on the state space; see Footnote \ref{fn-RND} for the definition of Radon-Nikodym derivative.} a stationary Markov perfect equilibrium may not exist. We introduce a new model called ``stochastic games with endogenous shocks'',\footnote{In contrast to the sunspot idea in \cite{DGMM1994} and \cite{NR1992}, the innovation of \cite{Duggan2012} is to allow the noise to be part of the stage payoffs. This feature is also shared by stochastic games with endogenous shocks as considered in this paper.} which allows the distribution of current period's shocks to directly depend on the ``discrete'' components of the states and actions from the previous period.\footnote{Dynamic economic models with some discrete components of actions and states are common, such as entry or exit by firms in \cite{EP1995} and \cite{Hopenhayn1992}.}

In order to prove an equilibrium existence result in general stochastic games that cover stochastic games with endogenous shocks and the class of games in which the state transitions have a decomposable feature (see, for example, \cite{Nowak2003} and \cite{PS1989}), we propose a condition called ``decomposable coarser transition kernels'' in the sense that the transition kernel is decomposed as a sum of finitely many components with each component being the product of a ``coarser" history-relevant transition function and a history-irrelevant density function.
For the case of stochastic games with endogenous shocks,  the history-irrelevant density functions represent the shocks. In particular, each discrete component of the states and actions from the previous period naturally contributes one component  in the transition kernel, which describes the history dependence of the distribution of the shock. It thus captures the intuition that the distribution of the random shock can directly depend on the actions and states in the previous period.

Our Theorem~\ref{thm-D_existence} shows that under the condition of a decomposable coarser transition kernel, a stochastic game has a stationary Markov perfect equilibrium. A very simple proof is given by providing a new link between a convexity type result of \cite{DE1976} on the conditional expectation of a correspondence and the equilibrium existence problem in stochastic games. As a corollary to Theorem~\ref{thm-D_existence}, we know that a stationary Markov perfect equilibrium exists in a stochastic game with endogenous shocks. Theorem \ref{thm-atom} extends Theorem~\ref{thm-D_existence} by including an atomic part in the transition probability, and  covers the main existence result in \cite{Nowak2003} as a special case. As an illustrative application of stochastic games with endogenous shocks, we consider a stochastic version of the dynamic oligopoly model as studied in \cite{MT1988a,MT1988b}.\footnote{Further applications will be  discussed in Remark \ref{rmk-ext}.} Note that our results on stochastic games with endogenous shocks and the stochastic dynamic oligopoly model cannot be covered by existing  results in the relevant literature, because the transition of the shock component in the state explicitly depends on the parameters in the previous stage.

To study a dynamic problem with a stationary structure, the standard approach is to work with a reduced problem with a recursive structure, in which players' payoffs are given by a convex combination of the stage payoffs and players' expected continuation values in  terms of the Bellman equations. To solve the existence problem for a stochastic game, one often needs to work with a one-shot auxiliary game parameterized by state variable $s$ and continuation value function $v$, in which the set of Nash equilibrium payoffs is denoted by  $P_v(s)$.  Though the correspondence $P_v(\cdot)$ is closed valued and upper hemicontinuous in terms of $v$, it is not convex valued in general. As a result, each of the desirable convexity, closedness and upper hemicontinuity properties would fail for the correspondence $R$ from the space of continuation value functions to itself, whose value at a continuation value function $v$ is the collection of measurable selections of the correspondence $P_v(\cdot)$. Hence, the classical Fan-Glicksberg Fixed-point Theorem (see Corollary 17.55 in  \cite{AB2006}) is not applicable to the correspondence $R$. Note that a fixed point of the correspondence $R$ will correspond to a stationary Markov perfect equilibrium of the stochastic game.\footnote{For details, see Subsection~\ref{subsec-coarser proof}.} On the other hand, the convex hull correspondence $\mbox{co}(R)$ of $R$ does have the above desirable properties. Thus, $\mbox{co}(R)$ has a fixed point, which corresponds to a stationary Markov perfect {\it correlated} equilibrium of the stochastic game as in \cite{DGMM1994} and \cite{NR1992}. The key insight of our proof is that this imposed convexity restriction can be relaxed by showing the equivalence of the conditional expectations of $\mbox{co}(R)$ and $R$ under the condition of a decomposable coarser transition kernel, which leads to a fixed point of the correspondence $R$ (instead of $\mbox{co}(R)$). The minimality of our condition for this ``one-shot game'' approach is also demonstrated from a technical point of view in the sense that the conditions are shown to be tight as in Propositions \ref{prop-atom} and \ref{prop-finite}.

The rest of the paper is organized as follows. Section~\ref{sec-game} presents a general model of discounted stochastic games. The main result is given in Section~\ref{sec-results}. Stochastic games with endogenous shocks and a stochastic dynamic oligopoly model are considered in Section~\ref{sec-application}. Section~\ref{sec-atomic} provides an  extension of the main existence result by allowing the transition kernel to have an atomic part. Section~\ref{sec-lite} discusses the relationship between our existence theorems and several related results in the literature, and also demonstrates the minimality of our condition. Section~\ref{sec-conclusion} concludes the paper. The Appendix collects the proofs.

\section{Discounted Stochastic Games}\label{sec-game}

An $m$-person discounted stochastic game can be defined in terms of (1) a state space, (2) a state-dependent feasible action correspondence for each player, (3) a stage-payoff for each player that depends on the state and action profile, (4) a discount factor for each player, and (5) a transition probability that depends on the state and action profile.
Formally, an $m$-person discounted stochastic game is described as follows:
\begin{itemize}
  \item $I=\{1,\cdots, m\}$ is the set of players.

  \item  $(S,\cS)$ is a measurable space representing the states of nature, where $\cS$ is countably generated.\footnote{It means that there is a countable subset $\cD$ of $\cS$ such that $\cS$ is generated by the sets in $\cD$. This condition is widely adopted. For example, the state spaces in \cite{DGMM1994}, \cite{Duggan2012}, and \cite{Levy2013} are respectively assumed to be a compact metric space, a Polish space, a Borel subset of a Polish space, and hence are all countably generated.}

  \item For each $i\in I$, $X_i$ is player~$i$' action space, which is a nonempty compact metric space endowed with its Borel $\sigma$-algebra $\cB(X_i)$. Let $X= \prod_{1\leq i \leq m}X_i$ and $\cB(X) = \otimes_{1 \le i \le m} \cB(X_i)$. Then $X$ is the space of all possible action profiles.

  \item For each $i\in I$, the set of feasible actions of player $i$ at state $s$ is given by $A_i(s)$, where $A_i$ is an $\cS$-measurable,\footnote{The correspondence $A_i$ is said to be $\cS$-measurable (resp. weakly $\cS$-measurable) if for any closed (resp. open) subset $B$ of $X_i$, the set $\{s \in S: A_i(s) \cap B \ne \emptyset\}$ belongs to $\cS$; see Definition 18.1 of \cite{AB2006}. In the literature, there are some papers which assume that the correspondence $A_i$ is weakly measurable; see, for example, \cite{Duggan2012} and \cite{NR1992}. These two measurability notions coincide in our setting since $A_i$ is compact valued; see Lemma~18.2 in \cite{AB2006}.} nonempty and compact valued correspondence from $S$  to $X_i$. Let $A(s) = \prod_{i\in I} A_i(s)$ for each $s \in S$.

  \item  For each $i\in I$, $u_i: S\times X\to \bR$ is player~$i$'s stage-payoff with an absolute bound $C$ (i.e., for every $i\in I$ and $(s,x) \in S\times X$, $|u_i(s,x)| \le C$ for some positive real number $C$) such that $u_i(s,x)$ is $\cS$-measurable in $s$ for each $x\in X$ and continuous in $x$ for each $s \in S$.

  \item $\beta_i\in[0,1)$ is player $i$'s discount factor.

  \item The law of motion for the states is given by the transition probability $Q: S\times X \times \cS \to [0,1]$.\footnote{Note that the payoff $u_i$ and the transition probability $Q$ only need to be defined on the graph of $A$. For simplicity, we follow the literature to define them on the whole product space $S \times X$, as in \cite{DGMM1994}, \cite{Levy2013}, \cite{LM2015} and \cite{NR1992}.} That is, if $s$ is the state at stage $t$ and $x\in X$ is the action profile chosen simultaneously by the $m$ players at this stage, then $Q(E|s,x)$ is the probability that the state at stage $t+1$ belongs to the set $E$ given $s$ and $x$.
     \begin{enumerate}
        \item $Q(\cdot|s,x)$ (abbreviated as $Q_{(s,x)}$) is a probability measure on $(S,\cS)$ for all $s \in S$ and $x\in X$, and for all $E\in \cS$, $Q(E|\cdot, \cdot)$ is $\cS \otimes \cB(X)$-measurable.

        \item For all $s$ and $x$, $Q(\cdot|s,x)$ is absolutely continuous with respect to a probability measure $\lambda$ on $(S,\cS)$.  Let $q$ be an $\cS \otimes \cS \otimes \cB(X)$-measurable function from $S \times S \times X$ to $\mathbb{R}_+$ such that for any $s \in S$ and $x\in X$, $q(\cdot,s,x)$ (also written as $q(\cdot|s,x)$ or $q_{(s,x)}$) is the corresponding Radon-Nikodym derivative of $Q(\cdot|s,x)$.\footnote{Let $(S, \cS, \lambda)$ be a probability space. A finite measure $\nu$ is said to be absolutely continuous with respect to $\lambda$ if for any $D \in \cS$, $\lambda (D) =0$ implies $\nu (D) =0$. In this case, there exists a ($\lambda$-almost) unique $\lambda$-integrable function $q$ such that $\nu(D) = \int_D q(s) \lambda(\rmd s)$ for any $D \in \cS$. Such a function $q$ is called the Radon-Nikodym derivative of $\nu$ with respect to $\lambda$, see p.134 in \cite{Loeb-16}. \label{fn-RND}}

        \item For all $s \in S$, the mapping $x \to Q(\cdot | s, x)$ is norm continuous; that is, for any sequence of action profiles $\{x^n\}$ converging to some $x^0$, the sequence $\{Q(\cdot | s, x^n) \}$ converges to $Q(\cdot | s, x^0)$ in total variation.\footnote{The total variation distance of two probability measures $\mu$ and $\nu$ on $(S, \cS)$ is $\left\|\mu -\nu \right\|_{{TV}} = \sup _{{D\in {\cS}}}|\mu (D)-\nu (D)|$. A sequence of probability measures $\{\mu_n\}$ is said to be convergent to a probability measure $\mu_0$ in total variation if $\lim_{n \to \infty} \left\|\mu_n -\mu_0 \right\|_{{TV}} = 0$.}

        \end{enumerate}
\end{itemize}

As in \cite{MT2001}, a history of a stochastic game up to stage $t\ge 1$ can be defined as follows. We use
$s_t$ and $\mathbf{x}_t$ to describe respectively the state and action profile in stage~$t$.
Let $h_1 = s_1 \in S$, $h_t = (s_1, \mathbf{x}_1, s_2, \ldots, \mathbf{x}_{t-1}, s_t)$ for $t \ge 2$, where $\mathbf{x}_{j-1} \in A(s_{j-1})$ and $s_j \in S$ for $2\le j \le t$, and $H_t$ be the space of all such $h_t$.
A strategy $f_i$ of player~$i$ specifies, for each stage $t\ge 1$, a measurable mapping (to be called a mixed action plan at stage $t$) from the space $H_t$ to the set of player~$i$'s mixed actions
$\cM(X_i)$,\footnote{For a Borel set $D$ in a complete separable metric space, let $\cM(D)$ be the space of all Borel probability measures on $D$.} which places probability $1$ on the set of feasible actions $A_i(s_t)$ at each state $s_t \in S$. For any profile of strategies $f = \{f_i\}_{i \in I}$ of the players and every initial state $s_1 \in S$, a probability measure $P_{s_1}^f$ is defined on $(S \times X)^{\infty}$ in a canonical way; see, for example, \cite{BS1978}. Given the strategy profile $f$ in the game starting from the state $s_1$, the expected payoff of player~$i$ is $\bE_{s_1}^f \left( \sum_{t = 1}^{\infty} \beta_i^{t-1} u_i(s_t, \mathbf{x}_t) \right)$, where the expectation is taken with respect to the probability measure $P_{s_1}^f$. A strategy for a player is called a Markov strategy if the mixed action plan specified for each stage $t\ge 1 $ only depends on the state $s_t \in S$. As highlighted in \cite{MT2001}, Markov strategies have the advantage that these strategies only depend on payoff-relevant data in the game.

In this paper, we shall focus on a particular class of Markov strategies, namely the ``stationary Markov strategies'', in which a player makes her decision based only on the current state but not the calendar time. Namely, the mixed action plan specified for each stage $t\ge 1 $ is the same mapping. Stationary Markov strategies are natural for the discounted payoff evaluation, since subgames starting at the same state are strategically equivalent in the sense that players have the same payoffs in these subgames. In addition, stationary Markov strategies are particularly useful because they are easy to analyze. Formally, a stationary Markov strategy for player~$i$ is an $\cS$-measurable mapping $f_i: S \to \cM(X_i)$ such that $f_i(s)$ places probability $1$ on the set $A_i(s)$ for each $s\in S$.\footnote{Since $A_i$ is measurable, nonempty and compact valued, the correspondence $A_i$ has a measurable graph by Theorem~18.6 in \cite{AB2006}. Then by Corollary~18.27 in \cite{AB2006}, $A_i$ has a measurable selection. As a result, the set of stationary Markov strategies of player~$i$ is nonempty for each $i$.} A stationary Markov strategy profile $f$ is called a \textbf{stationary Markov perfect equilibrium} if
$$\bE_{s_1}^f \left( \sum_{t = 1}^{\infty} \beta_i^{t-1} u_i(s_t,\mathbf{x}_t) \right) \ge \bE_{s_1}^{(g_i, f_{-i})} \left( \sum_{t = 1}^{\infty} \beta_i^{t-1} u_i(s_t,\mathbf{x}_t) \right)
$$
for any $i \in I$, $s_1 \in S$ and any strategy $g_i$ of player~$i$.

In the following, we shall consider stationary Markov perfect equilibria in terms of the recursive structure, which is much easier to work with. By standard results in dynamic programming and stochastic games (see, for example, \cite{Blackwell1965} and \cite{NR1992}), this formulation is equivalent to the equilibrium notion defined above. Given a stationary Markov strategy profile $f$, the continuation value $v(\cdot, f)$ gives an essentially bounded $\cS$-measurable mapping from $S$ to $\bR^m$, which is uniquely determined by the following recursion
\begin{equation} \label{eq-v_i}
v_i(s,f)=\int_{X} \left[(1-\beta_i)u_i(s,x)+\beta_i\int_S v_i(s_1,f) Q(\rmd s_1|s,x)\right]f(\rmd x|s).\footnote{The existence and uniqueness of the continuation value of each player follows from a standard contraction mapping argument, see \cite{Blackwell1965}.}
\end{equation}
The strategy profile $f$ is a \textbf{stationary Markov perfect equilibrium} if the discounted expected payoff of each player $i$ is maximized by her strategy $f_i$ in every state $s \in S$. It means that the continuation value $v$ solves the following recursive maximization problem:
\begin{eqnarray} \label{eq-sge}
v_i(s,f)
 =&&\max_{x_i\in A_i(s)} \int_{X_{-i}} \big[(1-\beta_i)u_i(s,x_i,x_{-i}) \nonumber \\
&& +\beta_i\int_S v_i(s_1,f) Q(\rmd s_1|s,x_i,x_{-i}) \big] f_{-i}(\rmd x_{-i}|s),
\end{eqnarray}
where $x_{-i}$ and $X_{-i}$ have the usual meanings, and $f_{-i}(s)$ is the product probability $\otimes_{j\neq i} f_j(s)$ on the product of the action spaces of all players other than player~$i$ at the state $s$.

\section{Main Result}\label{sec-results}

In this section, we show the existence of stationary Markov perfect equilibria for discounted stochastic games. In particular, we introduce the notion of ``(decomposable) coarser transition kernel'' and present our main result. Subsection~\ref{subsec-coarser proof} provides a simple proof via establishing a new connection between stochastic games and conditional expectations of correspondences.

Before moving to the statement of the condition and the result, we need a formal concept of an atom over a sub-$\sigma$-algebra. Let $\lambda$ be a probability measure on the measurable state space $(S,\cS)$. Let $\cG$ be a sub-$\sigma$-algebra of $\cS$. A set $D\in\cS$ of positive measure is said to be a $\cG$-atom if the restricted $\sigma$-algebras  of $\cG$ and $\cS$ to $D$ are essentially the same. When the relevant $\sigma$-algebras are used to represent information,\footnote{For example, in the context of a stochastic game, $\cS$ represents the total information about the state space while $\cG$ represents the information generated by the transition kernel $q$.} an event $D \in \cS$ is a $\cG$-atom simply means that given the realization of event $D$, $\cS$ and $\cG$ carry the same information. Formally,
let $\cG^D$ and $\cS^D$ be the respective $\sigma$-algebras $\{ D\cap D' \colon D' \in \cG \}$ and $\{ D\cap D' \colon D' \in \cS \}$ on $D$. The set $D \in \cS$ is said to be a $\cG$-atom\footnote{The notion of a $\cG$-atom has been considered in \cite{HN1966} and \cite{Neveu1967}, see also Definition~1.3.3 in \cite{Jacobs1978}. The authors of this paper are grateful to an anonymous referee for providing these references.} if the strong completion of $\cG^D$ is $\cS^D$.\footnote{The strong completion of $\cG^D$ in $\cS$ under $\lambda$ is the set of all sets in the form $E \triangle E_0$, where $E \in \cG^D$ and $E_0$ is a $\lambda$-null set in $\cS^D$, and $E \triangle E_0$ denotes the symmetric difference $(E\setminus E_0) \cup (E_0 \setminus E)$.}

\begin{defn}\label{defn-coarser}
A discounted stochastic game is said to have a coarser transition kernel if for some sub-$\sigma$-algebra $\cG$ of $\cS$, $q(\cdot|s,x)$ is $\cG$-measurable for all $s\in S$ and $x\in X$, and $\cS$ has no $\cG$-atom (under $\lambda$).\footnote{When $\cG$ is the trivial $\sigma$-algebra $\{S, \emptyset\}$, $\cS$ has no $\cG$-atom if and only if $\lambda$ is atomless. If $(S,\cS, \lambda)$ has an atom $D \in \cS$ in the usual sense (namely, $D$ is an atom over the trivial $\sigma$-algebra), then $D$ is automatically a $\cG$-atom for any sub-$\sigma$-algebra $\cG$ of $\cS$. It also means that if $\cS$ has no $\cG$-atom under $\lambda$, then $(S,\cS, \lambda)$ is atomless. }

A discounted stochastic game is said to have a decomposable coarser transition kernel if for some sub-$\sigma$-algebra $\cG$ of $\cS$, $\cS$ has no $\cG$-atom (under $\lambda$) and for some positive integer $J$,
$q(s_1|s,x) = \sum_{1\leq j\leq J}  q_j(s_1,s,x)\rho_j(s_1)$, where $q_j$ is $\cS \otimes \cS \otimes \cB(X)$-jointly measurable and $q_j(\cdot, s,x)$ is $\cG$-measurable for each $s\in S$ and $x\in X$, $q_j(\cdot, s,x)$ and $\rho_j$ are nonnegative and integrable on the atomless probability space $(S,\cS,\lambda)$, $j=1,\ldots,J$.

\end{defn}

Note that the condition of ``coarser transition kernel'' is a special case of the condition of ``decomposable coarser transition kernel''.

Below, we provide a simple example to illustrate the condition of ``a decomposable coarser transition kernel''.

\begin{exam} \label{ex-1}
Suppose that the state space is $S = [0,1] \times [0,1]$ (resp. $S_1 = [0,1] \times [0,1]$) with the generic element $s = (h, r)$ (resp. $s_1 = (h_1, r_1)$), $\lambda$ is the uniform distribution on the unit square, and the space of action profiles is $X \subseteq \bR^l$. For simplicity, the density function constructed below does not depend on $x \in X$.

Let
$$\rho_1(s_1) = h_1, \qquad \rho_2(s_1) = r_1;$$
$$q_1(s_1, s) = \frac{h_1 + h}{2/3 + h }, \qquad q_2(s_1, s) = \frac{2h_1 + r}{1 + r};$$
$$q(s_1|s) = q(s_1,s)= q_1(s_1,s)\rho_1(s_1) + q_2(s_1,s)\rho_2(s_1).$$
It is easy to check that the integration of $q(s_1, s)$ with respect to $s_1$ on $S_1$ is $1$ for any $s\in S$.

Let $\cS$ be the Borel $\sigma$-algebra on $[0,1] \times [0,1]$ and $\cG = \cB([0,1]) \otimes \{\emptyset, [0,1]\}$, where $\cB([0,1])$ is the Borel $\sigma$-algebra on $[0,1]$. For any $s \in S$, both $q_1(\cdot, s)$ and $q_2(\cdot, s)$ are linear functions of $h_1$ and do not depend on $r_1$. Thus, the $\sigma$-algebra generated by $\{q_1(\cdot, s), q_2(\cdot, s) \}_{s \in S}$ is $\cG$.\footnote{A $\sigma$-algebra is said to be generated by a collection of mappings taking values in some complete separable metric spaces if it is the smallest $\sigma$-algebra on which these mappings are measurable.} It is clear that $\cS$ has no $\cG$-atom.\footnote{For example, let $D$ be the subset $[0, 1/2] \times [0, 1]$ of $S$. Then, the set $[0, 1/2] \times [1/2, 1]$ is in $\cS^D$ but not in $\cG^D$, which means that $D$ is not a $\cG$-atom.} As a result, the density function $q$ satisfies the condition of a decomposable coarser transition kernel.

Notice that the functions $\{q(\cdot, s)\}_{s \in S}$ are measurable with respect to the $\sigma$-algebra $\cS$. Below, we show that $\cS$ is indeed generated by the collection of functions $\{q(\cdot, s)\}_{s \in S}$. The proof is left in the Appendix.
\end{exam}

\begin{claim} \label{claim-1}
In Example~\ref{ex-1}, the $\sigma$-algebra generated by $\{q(\cdot, s)\}_{s \in S}$ is $\cS$.
\end{claim}

The measurability requirements on the action correspondences, the stage payoffs, the transition probability and the transition kernel simply mean that the $\sigma$-algebra $\cS$ is generated by the four collections of mappings $\{A_i(\cdot)\}_{i \in I}$, $\{u_i(\cdot, x)\}_{i \in I, x \in X}$, $\{Q(E|\cdot, x)\}_{E \in \cS, x \in X}$ and $\{q(\cdot|s,x)\}_{s\in S, x\in X}$. When the transition kernel is in the form $q(s_1|s,x) = \sum_{1\leq j\leq J}  q_j(s_1,s,x)\rho_j(s_1)$, let $\cG$ be the $\sigma$-algebra generated by the mappings $\{q_j(\cdot|s,x)\}_{j\in J, s\in S,x\in X}$. In the context of Example \ref{ex-1}, if one works with some $\cS$-measurable action correspondences and stage payoffs, then the  mappings $\{A_i(\cdot)\}_{i \in I}$, $\{u_i(\cdot, x)\}_{i \in I, x \in X}$, $\{Q(E|\cdot, x)\}_{E \in \cS, x \in X}$ and $\{q(\cdot|s,x)\}_{s\in S, x\in X}$ generate the $\sigma$-algebra $\cS = \cB([0,1])\otimes \cB([0,1])$ (which is already generated by $\{q(\cdot|s,x)\}_{s\in S, x\in X}$, as shown in Claim \ref{claim-1}). On the other hand, The $\sigma$-algebra generated by $\{q_1(\cdot, s), q_2(\cdot, s) \}_{s \in S}$ is $\cG =\cB([0,1]) \otimes \{\emptyset, [0,1]\}$. The condition of a decomposable coarser transition kernel is satisfied in this case because $\cS$ has no $\cG$-atom. In particular, the word ``coarser'' here means that the $\sigma$-algebra $\cG$ induced by $\{q_1(\cdot, s), q_2(\cdot, s) \}_{s \in S}$ is coarser than the $\sigma$-algebra $\cS$ in the primitive of the game, given any non-trivial event.

We now state the main result of this paper.

\begin{thm}\label{thm-D_existence}
Every discounted stochastic game with a (decomposable) coarser transition kernel has a stationary Markov perfect equilibrium.
\end{thm}

\section{Applications}\label{sec-application}

In this section, we shall present some applications. In particular, we introduce in Subsection \ref{subsec-shock} a subclass of stochastic games with a decomposable coarser transition kernel, where the discrete components in the actions and states could directly influence the transition of the random shocks. Games in this subclass are called stochastic games with endogenous shocks.\footnote{As in \cite{Duggan2012}, we also allow the random shocks to enter into the stage payoffs of the players.} As an illustrative application, we shall consider in Subsection \ref{subsec-example} a stochastic version of a dynamic oligopoly model as studied in \cite{MT1988a,MT1988b}.

\subsection{Stochastic games with endogenous shocks}\label{subsec-shock}

Economic models with both discrete and continuous components of actions are common. For example, firms make discrete choices like entry or exit of the market, and continuous choices like quantities and prices of the products. Here we shall consider stochastic games where the discrete choices may play a distinct role. For this purpose, we assume that for each player $i\in I$, her action space has two components in the form $X_i = X_i^d \times X_i^{-}$, where $X_i^d$ is a finite set representing the discrete choices, and $X_i^{-}$ is a compact metric space representing possibly other types of actions. Let $X^d = \prod_{i\in I} X^d_i$ and $X^{-} = \prod_{i\in I} X_i^{-}$.

Since actions in the previous period could be part of the current state, the state space may have discrete and continuous components as well. On the other hand, random shocks form a common feature in various economic models. As a result, we assume that the state space has the form $S = H^d \times H^{-} \times R$. Here $R$ models the random shocks. The components $H^d$ and $H^{-}$ represent respectively a finite set of fundamental parameters and the residual part of fundamental parameters in the states.

As defined in Section~\ref{sec-game}, the state transition in a stochastic game generally depends on the actions and state in the previous period. However, as shown by the counterexample in \cite{LM2015}, a stationary Markov perfect equilibrium in a stochastic game may not exist without restriction on the transition kernel. Indeed the restriction, as considered in \cite{DGMM1994}, \cite{Duggan2012} and \cite{NR1992}, is to assume that the transition of random shocks in the stochastic games  does not depend on the actions and state in the previous period explicitly. For a stochastic game whose state space is in the form $S = H^d \times H^{-} \times R$, our innovation is to allow that the transition of random shocks directly depends on the discrete components $H^d$ and $X^d$ of the state and actions from the previous period.

We shall now describe formally the state space and the law of motion in a stochastic game with endogenous shocks with the rest of parameters and conditions as in Section~\ref{sec-game}.

\begin{enumerate}
  \item The state space can be written as $S = H^d \times H^{-} \times R$ and $\cS = \cH^d \otimes \cH^{-} \otimes \cR$, where $H^d$ is a finite set with its power set $\cH^d$, $H^{-}$ and $R$ are complete separable metric spaces endowed with the Borel $\sigma$-algebras $\cH^{-}$ and $\cR$ respectively. Denote $H = H^d\times H^{-}$ and $\cH = \cH^d \otimes \cH^{-}$.
  \item Recall that $Q: S\times X \times \cS \to [0,1]$ is the transition probability. For $s \in S, x \in X$, let $Q_H(\cdot|s, x)$ be the marginal of $Q(\cdot|s, x)$ on $H$. There is a fixed probability measure $\kappa$ on $(H,\cH)$ such that for all $s$ and $x$, $Q_H(\cdot|s, x)$ is absolutely continuous with respect to $\kappa$ with the corresponding product measurable Radon-Nikodym derivative $\phi(\cdot|s, x)$. For each $s \in S$, $Q_H(\cdot|s, x)$ is norm continuous in $x$.
  \item The distribution of $r' \in R$ in the current period depends on $h'\in H$ in the current  period as well as $h^d$ and $x^d$ in the previous period. In particular, $Q_R \colon H^d\times X^d \times H\times \cR \to [0, 1]$ is a transition probability such that for all $s = (h^d, h^{-}, r)$, $x = (x^d, x^{-})$, and $Z \in \cS$, we have
      $$Q(Z|s, x) = \int_H \int_R \mathbf{1}_Z(h', r') Q_R(\rmd r'|h^d, x^d, h')Q_H(\rmd h'|s, x).\footnote{For any set $A$, the indicator function $\mathbf{1}_{A}$ of $A$ is a function such that $\mathbf{1}_{A}(y)$ is $1$ if $y\in A$ and $0$ otherwise.}
      $$
  \item For any $h^d$, $x^d$ and $h'$, $Q_R(\cdot|h^d, x^d, h')$ is absolutely continuous with respect to an atomless probability measure $\nu$ on $(R, \cR)$ with the corresponding product measurable Radon-Nikodym derivative $\psi(\cdot|h^d, x^d, h')$.
\end{enumerate}

\begin{rmk} \label{rmk-sges}
A noisy stochastic game as considered in \cite{Duggan2012} is the case that $H^d$ and $X^d$ are singletons. When both $H^{-}$ and $X^{-}$ are singletons, the random shocks in the current period could fully depend on the action profile and the fundamental part of the state in the previous period since $Q_R$ could fully depend on $(h^d, x^d)$ in the previous period; see, for example, the stochastic dynamic oligopoly model in the next subsection. Remark \ref{rmk-ext} below considers possible applications in which $H^d$, $X^d$, $H^{-}$ and $X^{-}$ are all non-singletons.\footnote{Given a stochastic game with endogenous shocks $G$, one may define an auxiliary (noisy) stochastic game $G'$ with a new state space $S'$, where $S'$ expands the state space $S$ of $G$ by  including an additional component consisting of $(h^d, x^d)$ from the previous period. The main theorem in \cite{Duggan2012} implies the existence of a stationary Markov Perfect equilibrium in this auxiliary stochastic game $G'$. However, such an equilibrium strategy for $G'$ is not a stationary Markov Perfect equilibrium for $G$ because it depends on $(h^d, x^d)$ from the previous period in the original game (hence not a Markov strategy for $G$).}
\end{rmk}

The following result is a simple corollary of Theorem \ref{thm-D_existence}.

\begin{coro}\label{coro-shock}
A stochastic game with endogenous shocks has a decomposable coarser transition kernel, and hence possesses a stationary Markov perfect equilibrium.
\end{coro}

\begin{proof}
Let $\lambda = \kappa\otimes \nu$, and $\cG = \cH\otimes \{\emptyset, R\}$. Since $\nu$ is atomless, $\cS$ has no $\cG$-atom under $\lambda$. It is clear that for all $s = (h^d, h^{-}, r)$, $x = (x^d, x^{-})$, and $Z \in \cS$, we have
$$
Q(Z|s, x) = \int_{H \times R} \mathbf{1}_Z(h', r') \phi(h'|s,x) \cdot \psi(r'|h^d, x^d, h') \lambda(\rmd (h',r')),\footnote{For clarity and notational simplicity, we use $s$, $(h, r)$, and $(h^d, h^{-}, r)$ interchangably (similarly for their ``prime" versions).}
$$
which means that the corresponding transition kernel
$$q(s'|s, x) = \phi(h'|s,x) \cdot \psi(r'|h^d, x^d, h')$$
for $s' = (h',r')$.

For any fixed ${\tilde x}^d \in X^d$ and ${\tilde h}^d \in H^d$, define an $\cS \otimes \cS \otimes \cB(X)$-measurable
function  $\hat \phi_{\left({\tilde h}^d, {\tilde x}^d\right)}$ on $S \times S \times X$ and an $\cS$-measurable function $\hat \psi_{\left({\tilde h}^d, {\tilde x}^d\right)}$ on $S$ such that for any $s' = (h',r') \in S$, $s=(h^d, h^{-},r) \in S$ and $x = (x^d, x^{-}) \in X$,
\begin{itemize}
  \item $\hat \phi_{\left({\tilde h}^d, {\tilde x}^d\right)}(s',s,x) = \phi(h'|s,x)\cdot \mathbf{1}_{\{({\tilde h}^d, {\tilde x}^d)\}}\left((h^d, x^d)\right)$;

  \item $\hat \psi_{\left({\tilde h}^d, {\tilde x}^d\right)}(s') = \psi(r'|{\tilde h}^d,{\tilde x}^d, h')$.
\end{itemize}
It is clear that $\hat \phi_{\left({\tilde h}^d, {\tilde x}^d \right)}(\cdot,s,x)$ is $\cG$-measurable for any fixed $s$ and $x$.  Then we have
$$q(s'|s, x) = \phi(h'|s,x) \cdot \psi(r'|h',h^d, x^d) = \sum_{{\tilde h}^d \in H^d, {\tilde x}^d \in X^d} \hat\phi_{\left({\tilde h}^d, {\tilde x}^d\right)}(s',s,x)\cdot \hat\psi_{\left({\tilde h}^d, {\tilde x}^d\right)}(s').$$
Since $H^d$ and $X^d$ are both finite, the stochastic game with endogenous shocks has a decomposable coarser transition kernel, and a stationary Markov perfect equilibrium by Theorem \ref{thm-D_existence}.
\end{proof}

\subsection{Stochastic dynamic oligopoly with price competition}\label{subsec-example}

As mentioned in the previous subsection, it is natural to model certain economic situations as stochastic games with endogenous shocks. In this subsection, we shall consider a stochastic analog of a dynamic oligopoly asynchronous choice model as studied in \cite{MT1988a,MT1988b}. We verify that such a stochastic dynamic oligopoly model is actually a stochastic game with endogenous shocks. Corollary~\ref{coro-shock} then allows us to claim the existence of a stationary Markov perfect equilibrium.

Competition between two firms $(i = 1,2)$ takes place in discrete time with an infinite horizon. Each period begins with a set of firms active in the market, a vector of prices from the previous period, and a demand shock.  An active firm will make the price decision while an  inactive firm follows its price in the previous period. The firms' profits depend on the prices of both firms and the demand shock. This model is an extension of the model as considered in \cite{MT1988b}, where firms move alternatively. In contrast, we allow for any possible transition between inactive and active firms, and introduce endogenous demand shocks which is considered desirable in \cite[p.~587]{MT1988b}.

Formally, the set of players is $I = \{1,2\}$.  An element $\theta$ in the set $\Theta = \{\theta_1, \theta_2, \theta_3\}$ indicates which players are active at the beginning of each period. Player~$1$ (resp. player~$2$) is active if $\theta=\theta_1$ (resp. $\theta_2$); both players are active  if $\theta = \theta_3$. That is, instead of focusing on a fixed short commitment for two periods, we allow for random commitments (see \cite{MT1988a} for more discussion). As in the model of \cite[p.~573]{MT1988b}, the price space $P_1 = P_2 = P$ is assumed to be finite, which means that firms cannot set prices in units smaller than, say, a penny. Let $\tilde{P} = P$, and $\tilde{P}^2 = \tilde{P}\times \tilde{P}$ denote the set of prices from the previous period. Let $R$ be a set of demand shocks, which is a closed subset of the Euclidean space $\bR^l$. The state of the market is then summarized by a vector $(\theta, (\tilde{p_1}, \tilde{p_2}), r)$.

A decision of firm $i$ is to propose a price $a_i \in P_i$. At state $(\theta, (\tilde{p_1}, \tilde{p_2}), r)$, the set of feasible actions for firm~$i$ is $P_i$ if $\theta = \theta_i$ or $\theta_3$, and $\{\tilde{p_i}\}$ otherwise. If a firm is active, then it can pick any price in $P_i$. If the firm is inactive, then it must commit to its price from the previous period.

Given the state $s = (\theta, (\tilde{p_1}, \tilde{p_2}), r)$ and action profile $(p_1, p_2)$ in the previous period, the state in the current period, denoted by $(\theta', (\tilde{p_1}', \tilde{p_2}'), r')$, is determined as follows:
\begin{enumerate}
  \item $\theta'$ is determined by $s$ and $(p_1, p_2)$, following a transition probability $\kappa_1(\cdot|s, p_1, p_2)$;
  \item $\tilde{p_i}' = p_i$ for $i = 1,2$, which means that the action profile in the previous period is publicly observed and viewed as part of the state in the current period;
  \item $r'$ is  determined by an atomless transition probability $\mu(\cdot|\theta', p_1, p_2, \theta, \tilde{p_1}, \tilde{p_2})$, which means that demand shocks are directly determined by the prices and firms' relative positions in the previous and current periods.
\end{enumerate}

Suppose that the market demand function $D\colon P \times R \to \bR_+$ and the cost function $c \colon R\to \bR_+$ are both bounded. At the state $s = (\theta, (\tilde{p_1}, \tilde{p_2}), r)$, firm~$i$'s profit is given by
$$u_i(p_1,p_2,r) =
\begin{cases}
(p_i - c(r)) D(p_i,r)            & p_i < p_j;  \\
\frac{(p_i - c(r)) D(p_i,r)}{2}  & p_i = p_j; \\
0                                     & p_i > p_j.
\end{cases}
$$
where $(p_1,p_2)$ is the action profile in the current period. Firm~$i$ discounts the future by a discount factor $0 < \beta_i < 1$.

We shall show that this stochastic dynamic duopoly model can be viewed as a stochastic game with endogenous shocks. For both firms, $X_i^d = P_i$ and $X_i^{-}$ is a set with only one element. Let $H^d = \Theta\times \tilde{P}^2$, $H^{-}$ be a singleton set, and $H = H^d \times H^{-}$. Let $\kappa$ be the counting probability measure on $H$ and
$$\nu = \frac{1}{J}\sum_{(\theta, \theta') \in \Theta^2,  (p_1, p_2)\in P_1\times P_2, (\tilde{p_1},\tilde{p_2}) \in \tilde{P}^2} \mu(\cdot|\theta', p_1, p_2, \theta, \tilde{p_1}, \tilde{p_2}),$$
where $J$ is cardinality of the product space $\Theta^2\times P^4$. Then the marginal measure $Q_H(\cdot|s,p_1,p_2)$ is absolutely continuous with respect to $\kappa$, and $Q_R(\cdot|h',h^d,p_1,p_2) = \mu(\cdot|\theta', p_1, p_2, \theta, \tilde{p_1}, \tilde{p_2})$ is absolutely continuous with respect to $\nu$. Since the space of action profiles is finite, the continuity requirement on the payoff functions and transition kernel is automatically satisfied. Therefore, the following proposition follows from Corollary~\ref{coro-shock}.

\begin{prop} \label{prp-1}
The above dynamic duopoly model has a stationary Markov perfect equilibrium.
\end{prop}

\begin{rmk} \label{rmk-ext}
In the above model, the position of a firm (active or inactive) is randomly determined, and hence is not a choice variable. One can consider a dynamic market model in which the positions of the firms are determined by the endogenous entry/exit decisions; see, for example, \cite{Duggan2012}, \cite{EP1995} and \cite{Hopenhayn1992}. In particular, in the application on firm entry, exit, and investment in \cite{Duggan2012}, each firm needs to make decisions on entry/exit as well as on the production plan. Thus, the action space naturally has two components. The state space can be divided as a product of three parts: the first part $Z$ provides a list of firms which are present or absent in the market in the current period; the second part $K$ is the available capital stock in the current period; and the last part $R$ represents exogenously given random shocks which are i.i.d with a  density with respect to the Lebesgue measure. This was formulated in \cite{Duggan2012} as a noisy stochastic game, which can also be viewed as a stochastic game with endogenous shocks with $H^d$ as a singleton set and $H^{-} = Z\times K$. This model can be extended by letting $H^d =Z$, $H^{-} = K$, and the transition of $R$ depend on firms' positions in the previous and current periods. Such dependence of the random shocks allows for the interpretation that the change in the number of active firms as well as the shift of market positions by some firms\footnote{If $1$ and $0$ represent active or inactive firms respectively, then a shift of market position for the active (inactive) firm means to change its position from $1$ to $0$ ($0$ to $1$).}  do matter for the demand/supply shocks. For example, the decision of a big firm to quit the market is a shock that could significantly distort the expectation of the demand side, while the exit decision of a small firm may not even be noticed. The extended model remains a stochastic game with endogenous shocks, which has a stationary Markov perfect equilibrium.\footnote{As noted in the introduction, our existence results on stochastic games, whose state transition for the shock component explicitly depends on the parameters in the previous stage, cannot be covered by earlier results. These include Corollary \ref{coro-shock}, Proposition \ref{prp-1} as well as the results suggested in Remark~\ref{rmk-ext}.}
\end{rmk}

\section{An Extension}\label{sec-atomic}

In Section~\ref{sec-results}, we assume that the probability measure $\lambda$ is atomless on $(S,\cS)$.
Below, we consider the more general case that $\lambda$ may have atoms. To guarantee the existence of stationary Markov perfect equilibria, we still  assume the condition of decomposable coarser transition kernel, but only on the atomless part.

\begin{enumerate}
  \item There exist disjoint $\cS$-measurable subsets $S_a$ and $S_l$ such that $S_a\cup S_l = S$, $\lambda|_{S_a}$ is the atomless part of $\lambda$ while $\lambda|_{S_l}$ is the purely atomic part of $\lambda$. The subset $S_l$ is countable and each singleton $\{s_l\}$ with $s_l \in S_l$ is $\cS$-measurable with $\lambda(s_l) > 0$.\footnote{This assumption is only for simplicity. It is immediate to extend our result to the case that $S_l$ is a collection of at most countably many atoms in the usual measure-theoretic sense. Note that for a set $\{s\}$ with one element in $S$, we use $\lambda(s)$ to denote the measure of this set instead of $\lambda(\{s\})$.}
  \item For $s_a \in S_a$, the transition kernel $q(s_a|s,x) = \sum_{1\leq j\leq J} q_j(s_a,s,x) \rho_j(s_a)$ for some positive integer $J$, and for each $s \in S$ and $x \in X$,  where $q_j$ is product measurable, and $q_j(\cdot,s,x)$ and $\rho_j$ are nonnegative and integrable on the atomless measure space $(S_a,\cS^{S_a},\lambda|_{S_a})$ for $j=1,\ldots,J$.\footnote{It is clear that for any $E \in \cS$, the transition probability $Q(E|s, x) = \int_{E \cap S_a} q(s_a|s,x) \lambda(\rmd s_a) + \sum_{s_l \in S_l} \mathbf{1}_E (s_l) q(s_l|s,x) \lambda(s_l)$ for any $s \in S$ and $x \in X$.}
\end{enumerate}

\begin{defn}\label{defn-atom}
Let $\cG$ be a sub-$\sigma$-algebra of $\cS^{S_a}$. A discounted stochastic game is said to have a  \textbf{decomposable coarser transition kernel on the atomless part} if   $\cS^{S_a}$ has no $\cG$-atom under $\lambda|_{S_a}$ and  $q_j(\cdot, s,x)$ is $\cG$-measurable on $S_a$ for each $s\in S$ and $x\in X$, $j=1,\ldots,J$.
\end{defn}

The following theorem shows that the equilibrium existence result still holds.

\begin{thm}\label{thm-atom}
Every discounted stochastic game with a decomposable coarser transition kernel on the atomless part has a stationary Markov perfect equilibrium.
\end{thm}

\section{Discussion}\label{sec-lite}

In this section, we shall discuss the relationship between our results and several related results. In particular, we show that our results cover the existence results on correlated equilibria, noisy stochastic games, stochastic games with state-independent transitions, and stochastic games with mixtures of constant transition kernels as special cases. We also explicitly demonstrate why a recent counterexample fails to satisfy our conditions, and discuss the minimality of our conditions.

\paragraph{Correlated equilibria}\

It was proved in \cite{NR1992} that a stationary Markov perfect correlated equilibrium exists in discounted stochastic games in the setup described in our Section~\ref{sec-game}. Ergodic properties of such correlated equilibria were obtained in \cite{DGMM1994} under stronger conditions. They essentially assumed that players can observe the outcome of a public randomization device before making decisions at each stage.\footnote{For detailed discussions on such a public randomization device, or ``sunspot'', see \cite{DGMM1994} and their references.} Thus, the new state space can be regarded as $S'= S\times L$ endowed with the product $\sigma$-algebra $\cS' = \cS\otimes \cB$ and product measure $\lambda' = \lambda\otimes \eta$, where $L$ is the unit interval endowed with the Borel $\sigma$-algebra $\cB$ and Lebesgue measure $\eta$. Denote $\cG' = \cS\otimes \{\emptyset,L\}$.
Given $s',s'_1 \in S'$ and $x\in X$, the new transition kernel $q'(s'_1|s',x) = q(s_1|s,x)$, where $s$ (resp. $s_1$) is the projection of $s'$ (resp. $s'_1$) on $S$ and $q$ is the original transition kernel with the state space $S$. Thus, $q'(\cdot|s',x)$ is measurable with respect to $\cG'$ for any $s'\in S'$ and $x\in X$.
It is obvious from Lemma 2 in \cite{HSS2016} that $\cS'$ has no $\cG'$-atom. Then the condition of coarser transition kernel is satisfied for the extended state space $(S',\cS',\lambda')$,  and  the existence of a stationary Markov perfect equilibrium follows from Theorem~\ref{thm-D_existence}.\footnote{As noted in \cite{NR1992}, a stochastic version of Caratheodory's theorem implies that one can find a stationary Markov correlated equilibrium strategy as a stochastic convex combination of stationary Markov strategies.}
The drawback of this approach is that the ``sunspot'' is irrelevant to the fundamental parameters of the game. Our result shows that it can indeed enter the stage payoff $u$, the correspondence of feasible actions $A$ and the transition probability $Q$.

\paragraph{Stochastic games with finite actions and state-independent transitions}\

In \cite{PS1989}, they studied stochastic games with finite actions and state-independent atomless transitions. Namely, $X_i$ is finite for each $i \in I$ and the transition probability $Q$ does not directly depend on $s$ (to be denoted by $Q(\cdot|x)$). Let $\mbox{card}(X)$ be the cardinality of the finite set $X$ of action profiles. We shall check that such a stochastic game satisfies the condition of decomposable coarser transition kernels.

Let $\lambda$ be the probability measure ${1 \over \mbox{card}(X)} \sum_{x \in X}Q(\cdot|x)$. Then, for each $x$, $Q(\cdot|x)$ is absolutely continuous with respect to $\lambda$, and the corresponding Radon-Nikodym derivative is denoted by $q(\cdot|x)$ (abbreviated as $q_x$). For a fixed $x \in X$, let $\mathbf{1}_{\{x\}}$ be the indicator function of the singleton set $\{x\}$: $\mathbf{1}_{\{x\}}(y) = 1$ if $y = x$, and $0$ otherwise. Then we have
$$q(s'|x) = \sum_{y \in X} \mathbf{1}_{\{x\}} (y)q_y(s').$$
It is obvious that the condition of decomposable coarser transition kernels is satisfied with $\cG = \{\emptyset, S\}$.  Then a stationary Markov perfect equilibrium exists by Theorem~\ref{thm-D_existence}.

\paragraph{Decomposable constant transition kernels on the atomless part}\

Stochastic games with transition probabilities as combinations of finitely many measures on the atomless part were considered in \cite{Nowak2003}. In particular, the structure of the transition probability in \cite{Nowak2003} is as follows.

\begin{enumerate}
  \item $S_2$ is a countable subset of $S$ and $S_1 = S\setminus S_2$, each point in $S_2$ is $\cS$-measurable.
  \item There are atomless nonnegative measures $\mu_j$ concentrated on $S_1$, nonnegative measures $\delta_k$ concentrated on $S_2$, and measurable functions $q_j,b_k: S\times X \to [0,1]$, $1\le j \le J$ and $1\le k \le K$, where $J$ and $K$ are  positive integers. The transition probability $Q(\cdot|s,x) = \delta(\cdot|s,x) + Q'(\cdot|s,x)$ for each $s\in S$ and $x\in X$, where $\delta(\cdot|s,x) = \sum_{1\le k \le K} b_k(s,x) \delta_k(\cdot)$ and $Q'(\cdot|s,x) = \sum_{1\leq j\leq J} q_j(s,x) \mu_j(\cdot)$.
  \item For any $j$ and $k$, $q_j(s,\cdot)$ and $b_k(s,\cdot)$ are continuous on $X$ for any $s\in S$.
\end{enumerate}

We shall show that any stochastic game with the above structure satisfies the condition of decomposable coarser transition kernel on the atomless part.

Without loss of generality, assume that $\mu_j$ and $\delta_k$ are all probability measures. Let $\lambda (E) = \frac{1}{J+K} \big( \sum_{1\le j \le J} \mu_j(E) + \sum_{1\le k \le K} \delta_k (E) \big)$ for any $E\in \cS$. Then $\mu_j$ is absolutely continuous with respect to $\lambda$ and assume that $\rho_j$ is the Radon-Nikodym derivative for $1\le j \le J$.

Given any $s\in S$ and $x\in X$, let
$$q(s'|s,x) = \begin{cases}
\sum_{1\le j \le J} q_j(s,x)\rho_j(s'), & \mbox{if } s' \in S_1; \\
\frac{\delta(s'|s,x)}{\lambda(s')},     & \mbox{if } s' \in S_2 \mbox{ and } \lambda(s') > 0;\\
0,                                      & \mbox{if } s' \in S_2 \mbox{ and } \lambda(s') = 0.
\end{cases}
$$
Then $Q(\cdot|s,x)$ is absolutely continuous with respect to $\lambda$ and $q(\cdot|s,x)$ is the transition kernel.
The condition of a decomposable coarser transition kernel on the atomless part is satisfied with $\cG = \{\emptyset, S_1\}$.  Then a stationary Markov perfect equilibrium exists by Theorem~\ref{thm-atom}.

\paragraph{Noisy stochastic games}\

It was proved in \cite{Duggan2012} that stationary Markov perfect equilibria exist in stochastic games with noise -- a component of the state that is nonatomically distributed and not directly affected by the previous period's state and actions. As indicated in Remark~\ref{rmk-sges}, a noisy stochastic game is a special case of a stochastic game with endogenous shocks such that $H^d$ and $X^d$ are both singletons.

By Corollary~\ref{coro-shock}, a noisy stochastic games has a decomposable coarser transition kernel. Below, we directly show that any noisy stochastic game indeed has a coarser transition kernel. The existence of stationary Markov perfect equilibria in noisy stochastic games thus follows from Theorem~\ref{thm-D_existence}. The proof of the proposition below is left in the appendix.

\begin{prop}\label{prop-noisy}
Every noisy stochastic game has a coarser transition kernel.
\end{prop}

\paragraph{Nonexistence of stationary Markov perfect equilibrium}\

A concrete example of a stochastic game satisfying all the conditions as stated in Section \ref{sec-game} was presented in \cite{LM2015}, which has no stationary Markov perfect equilibrium. Their example will be described in the following.

\begin{enumerate}
\item The set of players is $\{A, B, C, C', D, D', E, F\}$.
\item Player A has the action space $\{U, D\}$, and player $B$ has the action space $\{L, M, R\}$. Players $C, C', D, D'$ have the same action space  $\{0, 1\}$. Players E, F have the action space $\{-1, 1\}$.
\item The state space is $S = [0, 1]$ endowed with the Borel $\sigma$-algebra $\cB$.
\item For any action profile $x$, let
$$h(x) = x_C + x_{C'}  + x_D + x_{D'},$$
where $x_i$ is the action of player~$i$. For each $s \in [0,1]$, let
$$\tilde{Q}(s, x) = (1 - s) \cdot \frac{1}{64} h(x),$$
and $U(s,1)$ be the uniform distribution on $[s,1]$ for $s\in [0,1)$. The transition probability is given by
$$Q(s, x) = \tilde{Q}(s, x)U(s,1) + (1 - \tilde{Q}(s,x)) \delta_1,$$
where $\delta_1$ is the Dirac measure at $1$.
\end{enumerate}

The following proposition shows that the condition of a decomposable coarser transition kernel on the atomless part is violated in this example.\footnote{Proposition~\ref{prop-levy} can also be implied by the nonexistence result in \cite{LM2015} and our Theorem~\ref{thm-atom}. However, the argument in \cite{LM2015} is  deep, and our proof explicitly demonstrates why their example fails to satisfy our sufficient condition in Theorem~\ref{thm-atom}.}

\begin{prop}\label{prop-levy}
The atomless part $\tilde{Q}(s, x)U(s,1)$ of the transition probability in the Example of \cite{LM2015} does not have a decomposable coarser transition kernel.
\end{prop}

The proof will be given in Subsection \ref{sub-levy}. Here we provide some intuition why the decomposable coarser transition kernel condition fails in this example. Suppose that the atomless part $\tilde{Q}(s, x)U(s,1)$ ($s \in [0, 1)$) of the transition probability satisfies the condition with respect to the Lebesgue measure $\eta$ on $[0, 1)$.
The Radon-Nikodym derivative of $U(s,1)$  with respect to $\eta$ on $[0, 1)$ is:
$q(s_1|s) =
\begin{cases}
\frac{1}{1-s} & s_1\in [s,1), \\
0             & s_1\in [0,s).
\end{cases}
$
Given the particular form of $\tilde{Q}(s, x)$, we can claim that for some positive integer $J$,
$q(\cdot|s) = \sum_{1\le j \le J} q_j(\cdot,s) \rho_j(\cdot)$ for any $s\in [0,1)$, where $q_j$ is product measurable, $q_j(\cdot,s)$ and $\rho_j (\cdot)$ are nonnegative and $\eta$-integrable. Moreover, for some sub-$\sigma$-algebra $\cG$ of the Borel $\sigma$-algebra $\tilde{\cS}$ on $[0, 1)$, $q_j(\cdot, s)$ is $\cG$-measurable for each $j, s$, and $\tilde{\cS}$ has no $\cG$-atom.
Consider the simple case that $\rho_j$ is strictly positive on $[0, 1)$ for all $j$. For any fixed $s \in [0, 1)$, let $D_s = \{s_1 \in [0, 1): q_j(s_1,s) = 0, \, \forall j = 1, \ldots, J\}$, which is $\cG$-measurable. If $s_1 \in [0, s)$, then $q(s_1|s) =0$, which implies that $q_j(s_1,s) = 0$ for all  $1\le j \le J$; hence $s_1 \in D_s$. If $s_1 \in [s, 1)$, then $q(s_1|s) ={1 \over 1-s}$, which implies that $q_j(s_1,s) \ne 0$ for some  $1\le j \le J$; hence $s_1 \notin D_s$. Therefore, $D_s = [0, s)$ is in $\cG$. By the arbitrary choice of $s\in [0,1)$, we know that $\cG$ is the same as $\tilde{\cS}$ which is generated by the class of intervals \{$[0, s)\}_{s\in [0,1)}$. It means that $\tilde{\cS}$ has a $\cG$-atom $[0, 1)$, which is a contradiction.

\paragraph{Minimality}\

We have shown the existence of  stationary Markov perfect equilibria in discounted stochastic games by assuming the presence of a (decomposable) coarser transition kernel. This raises the question of whether our condition is minimal and, if so, then in what sense.

As discussed in the introduction, the main difficulty in the existence argument for stochastic games is due to the failure of the convexity of the equilibrium payoff correspondence $P$ of a one-shot auxiliary game as parameterized by state $s$ and  continuation value function $v$. As will be shown in Subsection~\ref{subsec-coarser proof}, the correspondence $R(v)$, which is the collection of  selections from the equilibrium payoff correspondence $P(v, \cdot)$, will live in an infinite-dimensional space if there is a continuum of states. Thus, the desirable closedness and upper hemicontinuity properties would fail even though $P$ has these properties in terms of $v$. To handle such issues, the standard approach is to work with the convex hull $\mbox{co}(R)$. We bypass this imposed convexity restriction by using the result that for any $\cS$-measurable, integrably bounded\footnote{A correspondence $G \colon (\cS,\cS,\lambda) \to \bR^n$ is said to be integrably bounded if there exists some integrable function $\varphi \colon (\cS,\cS,\lambda) \to \bR_+$ such that $\|G(s)\| \le \varphi(s)$ for $\lambda$-almost all $s$, where $\|\cdot\|$ is usual norm on $\bR^n$.} and closed valued correspondence $G$, if $\cS$ has no $\cG$-atom, then $\cI^{(\cS,\cG,\lambda)}_G = \cI^{(\cS,\cG,\lambda)}_{\mbox{co} (G)}$, where $\cI^{(\cS,\cG,\lambda)}_{G}=\{\bE^{\lambda}(g|\cG)\colon g\mbox{ is an } \cS\mbox{-measurable selection of } G \}$  ($\cI^{(\cS,\cG,\lambda)}_{\mbox{co} (G)}$ is defined analogously) and the conditional expectation is taken with respect to $\lambda$.
Moreover, for the condition of a decomposable coarser transition kernel, we assume that the transition kernel can be  divided into finitely many parts.
The following propositions demonstrate the minimality of our condition from a technical point of view.

\begin{prop}\label{prop-atom}
Suppose that $(S,\cS, \lambda)$ has a $\cG$-atom $D$ with $\lambda(D) > 0$. Then there exists a measurable correspondence $G$ from $(S,\cS,\lambda)$ to $\{0,1\}$ such that $\cI^{(\cS,\cG,\lambda)}_G \neq \cI^{(\cS,\cG,\lambda)}_{\mbox{co} (G)}$.
\end{prop}

The key result that we need in the proof of Theorem \ref{thm-D_existence} in Subsection~\ref{subsec-coarser proof} is 
the existence of an $\cS$-measurable selection $v^*$ of a correspondence $G$ 
such that  $\bE^{\lambda}(v^*\rho_j|\cG) = \bE^{\lambda}(v'\rho_j|\cG)$ for each $1\leq j \leq J$, where $v'$ is an $\cS$-measurable selection of $\mbox{co} (G)$, and $\{\rho_j\}_{1\leq j \leq J}$ are the functions as in Definition \ref{defn-coarser} for a decomposable coarser transition kernel. The question is whether a similar result holds if we generalize the condition of a decomposable coarser transition kernel from a finite sum to a countable sum. We will show that this is not possible.

Let $(S,\cS,\lambda)$  be the Lebesgue unit interval $(L,\cB,\eta)$.  Suppose that $\{\varrho_n\}_{n\ \ge 0}$ is a complete orthogonal system in $L_2(S,\cS,\lambda)$ such that $\varrho_n$ takes value in $\{-1,1\}$ and $\int_S \varrho_n \rmd \lambda = 0$ for each $n \ge 1$ and $\varrho_0 \equiv 1$.
Let $\rho_n  = \varrho_n + 1$ for each $n\ge 1$ and $\rho_0 = \varrho_0$.
Let $\{E_n\}_{n\geq 0}$ be a countable measurable partition of $S$, $E_n$ nonempty and $q_n(s) = \mathbf{1}_{E_n}$ for each $n\ge 0$. Suppose that a transition kernel $q$ is decomposed into a countable sum $q(s_1|s,x) = \sum_{n\ge 0} q_n(s)\rho_n(s_1)$. The following proposition shows that the argument for the case that $J$ is finite is not valid for such an extension.\footnote{It is a variant of a well known example of Lyapunov.}

\begin{prop}\label{prop-finite}
Let $G(s) =  \{-1,1\}$ for $s \in S$, and $f$ be the measurable selection of $\mbox{co}(G)$ that takes the constant value $0$. Then, for any sub-$\sigma$-algebra $\cF \subseteq \cS$, there is no $\cS$-measurable selection $g$ of $G$ with  $\bE^{\lambda}(g \rho_n|\cF) = \bE^{\lambda}(f \rho_n|\cF)$ for any $n\ge 0$.
\end{prop}

Thus, our condition is minimal in the sense that if one would like to adopt the ``one-shot game'' approach as used in the literature for obtaining a stationary Markov perfect equilibrium, then it is a tight condition.

\section{Concluding Remarks}\label{sec-conclusion}

We consider stationary Markov perfect equilibria in discounted stochastic games with a general state space. So far, such equilibria have been shown to exist only for several  special classes of stochastic games.  In this paper, the existence of stationary Markov perfect equilibria is proved under some general conditions, which broadens the scope of potential applications of stochastic games. We illustrate such applications via two examples, namely, stochastic games with endogenous shocks and a stochastic dynamic oligopoly model. Our results unify and go beyond various existence results as in \cite{DGMM1994}, \cite{Duggan2012}, \cite{Nowak2003}, \cite{NR1992} and \cite{PS1989}, and also provide some explanation for a recent counterexample which fails to have an equilibrium in stationary strategies.

In the literature, the standard recursive approach for the existence arguments is to work with
a one-shot auxiliary game parameterized by the state and the continuation value function.
We adopt this approach and provide a very simple proof under the condition of ``(decomposable) coarser transition kernels''. The proof is based on a new connection between stochastic games and conditional expectations of correspondences. We demonstrate from a technical point of view that our condition is minimal for the standard ``one-shot game'' approach.

\section{Appendix}\label{sec-appendix}

\subsection{Proof of Claim~\ref{claim-1}}

 Let $\cF$ be the sub-$\sigma$-algebra of $\cS$ that is generated by the collection of functions $\{q(\cdot, s)\}_{s \in S}$. Let $(\hat{h},\hat{r}) = (0,0)$ (resp. $(\tilde{h},\tilde{r}) = (\frac{1}{3}, 0)$), and denote $\hat{q}$ (resp. $\tilde{q}$) as the corresponding value of $q(s_1, \hat{h},\hat{r})$ (resp. $q(s_1, \tilde{h},\tilde{r})$). Then both $\hat{q}$ and $\tilde{q}$ are functions of $s_1$ and $\cF$-measurable. We have the following system of equations:
\begin{eqnarray}
\hat{q} &=& \frac{3}{2}h_1^2 + 2h_1r_1, \label{eq-ex1-1}\\
\tilde{q} &=& h_1^2 + \frac{1}{3}h_1 + 2h_1r_1. \label{eq-ex1-2}
\end{eqnarray}
We can view $h_1$ and $r_1$ as two functions of $(\hat{q}, \tilde{q})$. By subtracting  Equation (\ref{eq-ex1-2}) from Equation (\ref{eq-ex1-1}), we get $\hat{q} - \tilde{q} = \frac{1}{2}h_1^2 - \frac{1}{3}h_1$. By the definition, this equation must have  solutions in terms of $h_1$ for the given value $\hat{q} - \tilde{q}$. Thus, we have $\frac{1}{9} + 2(\hat{q} - \tilde{q}) \ge 0$, and $h_1 =  \frac{1}{3} + \sqrt{\frac{1}{9} + 2(\hat{q} - \tilde{q})}$ or $ \frac{1}{3} - \sqrt{\frac{1}{9} + 2(\hat{q} - \tilde{q})}$, with at least one of them in $[0,1]$. We can denote
$$\hat{\alpha}(\hat{q}, \tilde{q}) =
\begin{cases}
\frac{1}{3} + \sqrt{\frac{1}{9} + 2(\hat{q} - \tilde{q})}, \quad  \mbox{if } \frac{1}{3} + \sqrt{\frac{1}{9} + 2(\hat{q} - \tilde{q})} \in [0,1]; \\
\frac{1}{3} - \sqrt{\frac{1}{9} + 2(\hat{q} - \tilde{q})}, \quad \mbox{otherwise}.
\end{cases}$$
By substituting $h_1 = \hat{\alpha}(\hat{q}, \tilde{q})$ into Equation (\ref{eq-ex1-1}), we can solve $r_1$ as a function of $(\hat{q}, \tilde{q})$, which is denoted by $r_1 = \tilde{\alpha}(\hat{q}, \tilde{q})$.\footnote{By the construction, both $\hat{\alpha}$ and $\tilde{\alpha}$ are obviously Borel measurable functions.}

Let $\pi_h$ and $\pi_r$ be the projection mappings on $S_1$ with $\pi_h (h_1, r_1) = h_1$ and $\pi_h (h_1, r_1) = r_1$ respectively. Because both $\hat{q}$ and $\tilde{q}$ are $\cF$-measurable functions on $S_1$, and $\pi_h = \hat{\alpha}(\hat{q}, \tilde{q})$ and $\pi_r = \tilde{\alpha}(\hat{q}, \tilde{q})$, the mappings $\pi_h$ and $\pi_r$ are also $\cF$-measurable. As a result, both $\cB([0,1]) \otimes \{\emptyset, [0,1]\}$ and $\{\emptyset, [0,1]\} \otimes \cB([0,1])$ are contained in $\cF$. Thus, $\cS \subseteq \cF$. Since $\cF \subseteq \cS$ by the definition of $\cF$, we have $\cF = \cS$.

\subsection{Proof of Theorem~\ref{thm-D_existence}}\label{subsec-coarser proof}

In this subsection, we shall prove Theorem~\ref{thm-D_existence}.

Let $L_1((S,\cS,\lambda),\bR^m)$ and $L_{\infty}((S,\cS,\lambda),\bR^m)$  be the $L_1$ and $L_\infty$ spaces of all $\cS$-measurable mappings from $S$ to $\bR^m$ with the usual norm; that is,
$$L_1((S,\cS,\lambda),\bR^m) = \{f\colon f \mbox{ is } \cS \mbox{-measurable and } \int_S \|f\| \rmd \lambda   < \infty  \},
$$
$$L_{\infty}((S,\cS,\lambda),\bR^m) = \{f\colon f \mbox{ is } \cS \mbox{-measurable and essentially bounded under } \lambda \},
$$
where $\|\cdot\|$ is the usual norm in $\bR^m$.  By the Riesz representation theorem (see Theorem~13.28 of \cite{AB2006}), $L_{\infty}((S,\cS,\lambda),\bR^m)$ can be viewed as the dual space of $L_{1}((S,\cS,\lambda),\bR^m)$. Then $L_{\infty}((S,\cS,\lambda),\bR^m)$ is a locally convex, Hausdorff topological vector space under the weak$^*$ topology. Let $V = \{v \in L_{\infty}((S,\cS,\lambda),\bR^m) \colon \|v\|_\infty\leq C\}$, where $C$ is the upper bound of the stage payoff function $u$ and $\|\cdot \|_\infty$ is the essential sup norm of $L_{\infty}((S,\cS,\lambda),\bR^m)$. Then $V$ is nonempty and convex. Moreover, $V$  is compact under the weak$^*$ topology by Alaoglu's Theorem (see Theorem~6.21 of \cite{AB2006}). Since $\cS$ is countably generated, $L_1((S,\cS,\lambda),\bR^m)$ is separable, which implies that $V$ is metrizable in the weak$^*$ topology (see   Theorem~6.30 of \cite{AB2006}).

Given any $v=(v_1,\cdots, v_m)\in V$ and $s \in S$, we consider the game $\Gamma(v, s)$. The action space for player $i$ is $A_i(s)$. The payoff of player $i$ with the action profile $x\in A(s)$ is given by
\begin{equation} \label{eq-game}
U_i(s,x)(v)=(1-\beta_i)u_i(s,x)+\beta_i\int_S v_i(s_1) Q(\rmd s_1|s,x).
\end{equation}
A mixed strategy of player $i$ is an element in $\cM(A_i(s))$, and a mixed strategy profile is an element in $\bigotimes_{i\in I}\cM(A_i(s))$.
The set of mixed strategy Nash equilibria of the static game $\Gamma(v,s)$, denoted by $N(v,s)$, is a nonempty compact subset of $\bigotimes_{i\in I}\cM(X_i)$ under the weak$^*$ topology due to the Fan-Glicksberg Theorem (see \cite{Glicksberg1952} and Corollary~17.55 in \cite{AB2006}). Let $P(v,s)$ be the set of payoff vectors  induced by the Nash equilibria in $N(v,s)$, and $\mbox{co}(P)$ the convex hull of $P$. Then $\mbox{co}(P)$ is a correspondence from $V\times S$ to $\bR^m$. Let $R(v)$ (resp. $\mbox{co}(R(v))$) be the set of $\lambda$-equivalence classes of $\cS$-measurable selections of $P(v,\cdot)$ (resp. $\mbox{co}(P(v,\cdot))$) for each $v\in V$.

Following the arguments in Lemmas~6 and 7 in \cite{NR1992} (see also \cite{MP1987}),\footnote{In \cite{NR1992}, a slightly stronger condition was imposed on the transition probability $Q$ that the mapping $q(\cdot|s,x)$ satisfies the $L_1$ continuity condition in $x$ for all $s \in S$. Their arguments on the convexity, compactness and upper hemicontinuity properties still hold in our setting.} for each $v\in V$,  $P(v,\cdot)$ (abbreviated as $P_v(\cdot)$) is $\cS$-measurable and compact valued, and $\mbox{co}(R(v))$ is nonempty, convex, weak$^*$ compact valued and upper hemicontinuous. Then the correspondence $\mbox{co}(R): V\to V$ maps the nonempty, convex, weak$^*$ compact set $V$ (a subset of a locally convex Hausdorff topological vector space) to nonempty, convex subsets of $V$, and it has a closed graph in the weak$^*$ topology. By the classical Fan-Glicksberg Fixed Point Theorem, there is a fixed point  $v' \in V$ such that $v' \in \mbox{co}(R) (v')$. That is, $v'$ is an $\cS$-measurable selection of $\mbox{co}(P(v',\cdot))$.

Recall that a correspondence $G \colon S \to \bR^n$ is said to be integrably bounded if there exists some integrable function $\varphi$ such that $\|G(s)\| \le \varphi(s)$ for $\lambda$-almost all $s$. In addition, for any integrably bounded correspondence $G$ from $S$ to $\bR^m$, $\cI^{(\cS,\cG,\lambda)}_{G}=\{\bE^{\lambda}(g|\cG)\colon g\mbox{ is an } \cS\mbox{-measurable selection of } G \}$, where the conditional expectation is taken with respect to $\lambda$.

The following lemma is from \cite[Theorem 1.2]{DE1976}.

\begin{lem}\label{lem-convexity}
If  $\cS$ has no $\cG$-atom,\footnote{In \cite{DE1976}, a set $D\in \cS$ is said to be a $\cG$-atom if $\lambda(D)>0$ and given any  $D_0 \in \cS^D$,  $\lambda \big(s\in S: 0<\lambda(D_0\mid\cG) (s) <\lambda(D\mid\cG) (s) \big) =0 $. The conditions that $\cS$ has no $\cG$-atom as in \cite{DE1976} as well as in our paper and the book \cite{Jacobs1978}
are equivalent; see Lemma~2 in \cite{HSS2016}.} then for any  $\cS$-measurable,\footnote{In \cite{DE1976}, the correspondence $G$ is said to be measurable if for any $x \in \bR^m$, the function $d(x, G(s))$ is measurable, where $d$ is he usual metric in the Euclidean space. Their notion of measurability of a correspondence coincides with our definition of the measurability of a correspondence, see Theorem~18.5 in \cite{AB2006}.} $\lambda$-integrably bounded, closed valued correspondence $G$,  $\cI^{(\cS,\cG,\lambda)}_G = \cI^{(\cS,\cG,\lambda)}_{\mbox{co} (G)}$.
\end{lem}

Now we are ready to prove Theorem~\ref{thm-D_existence}.

\begin{proof}
Given $v'$, let
$$H(s) = \{(a, a \cdot \rho_1(s), \ldots, a \cdot \rho_J(s)) \colon  a \in P_{v'}(s) \},$$
and $\mbox{co}(H(s))$ the convex hull of $H(s)$ for each $s\in S$.
It is clear that $H$ is $\cS$-measurable, $\lambda$-integrably bounded and closed valued. Then $\cI^{(\cS,\cG,\lambda)}_{H} = \cI^{(\cS,\cG,\lambda)}_{\mbox{co}(H)}$ by Lemma \ref{lem-convexity}, which implies that there exists an $\cS$-measurable selection $v^*$ of $P_{v'}$ such that  $\bE^{\lambda}(v^*\rho_j|\cG) = \bE^{\lambda}(v'\rho_j|\cG)$ for each $1\leq j \leq J$.
For each $i\in I$,  $s\in S$ and $x\in X$, we have
\begin{align*}
\int_S v^*_i(s_1) Q(\rmd s_1|s,x)
& = \sum_{1\leq j \leq J} \int_S v^*_i(s_1) q_j(s_1,s,x) \rho_j(s_1) \lambda(\rmd s_1) \\
& = \sum_{1\leq j \leq J} \int_S \bE^{\lambda}(v^*_i \rho_j q_j(\cdot,s,x)|\cG)(s_1) \lambda(\rmd s_1)\\
& = \sum_{1\leq j \leq J} \int_S \bE^{\lambda}(v^*_i \rho_j|\cG)(s_1) q_j(s_1,s,x) \lambda(\rmd s_1)\\
& = \sum_{1\leq j \leq J} \int_S \bE^{\lambda}(v'_i \rho_j|\cG)(s_1) q_j(s_1,s,x) \lambda(\rmd s_1)\\
& = \sum_{1\leq j \leq J} \int_S \bE^{\lambda}(v'_i \rho_j q_j(\cdot,s,x)|\cG)(s_1) \lambda(\rmd s_1)\\
& = \sum_{1\leq j \leq J} \int_S v'_i(s_1) q_j(s_1,s,x) \rho_j(s_1) \lambda(\rmd s_1) \\
& = \int_S v'_i(s_1) Q(\rmd s_1|s,x).
\end{align*}
By Equation~(\ref{eq-game}),  $\Gamma(v^*,s) = \Gamma(v',s)$ for any $s\in S$, and hence $P(v^*,s) = P(v',s)$. Thus, $v^*$ is an $\cS$-measurable selection of $P_{v^*}$.

By the definition of $P_{v^*}$, these exists an $\cS$-measurable mapping $f^*$ from $S$ to $\bigotimes_{i\in I}\cM(X_i)$ such that $f^*(s)$ is a mixed strategy Nash equilibrium of the game $\Gamma(v^*,s)$ and $v^*(s)$ is the corresponding equilibrium payoff for each $s\in S$.\footnote{Note that $v^*$ is indeed the corresponding equilibrium payoff for $\lambda$-almost all $s\in S$. However, one can modify $v^*$ on a null set such that the claim holds for all $s\in S$; see, for example, \cite{NR1992}.}
It is clear that Equations~(\ref{eq-v_i}) and (\ref{eq-sge}) hold for $v^*$ and $f^*$, which implies that $f^*$ is a stationary Markov perfect equilibrium.
\end{proof}

\subsection{Proof of Theorem~\ref{thm-atom}}

Let  $V_a$ be the set of $\lambda$-equivalence classes of $\cS$-measurable mappings from $S_a$ to $\bR^m$ bounded by $C$.
For each $i\in I$, let $F_i$ be the set of all $f_i \colon S_l \to \cM(X_i)$ such that $f_i(s) \big( A_i(s) \big) = 1$ for all $s\in S_l$, $F=\prod_{i\in I}F_i$. Let $V_l$ be the set of mappings from $S_l$ to $\bR^m$ bounded by $C$, which is endowed with the supremum metric and hence a complete metric space.

Given $s\in S$, $v^a \in V_a$ and $v^l \in V_l$, consider the game $\Gamma(v^a, v^l, s)$. The action space for player $i$ is $A_i(s)$. The payoff of player $i$ with the action profile $x\in A(s)$ is given by
\begin{eqnarray} \label{eq-Phi_i}
\Phi_i(s,x,v^a,v^l)
& = (1-\beta_i) u_i(s,x) + \beta_i \sum_{1\leq j\leq J} \int_{S_a} v^a_i(s_a) q_j(s_a|s,x) \rho_j(s_a) \lambda(\rmd s_a) \nonumber \\
& + \beta_i \sum_{s_l \in S_l} v^l_i(s_l) q(s_l|s,x)\lambda(s_l).
\end{eqnarray}
The set of mixed-strategy Nash equilibria in the game $\Gamma(v^a, v^l, s)$ is denoted as $N(v^a, v^l, s)$. Let $P(v^a, v^l, s)$ be the set of payoff vectors induced by the Nash equilibria in $N(v^a, v^l,s)$, and $\mbox{co}(P)$ the convex hull of $P$.

Given $v^a\in V_a$, $f\in F$, define a mapping $\Pi$ from $V_l$ to $V_l$ such that for each $i\in I$, $v^l \in V_l$ and $s_l\in S_l$,
\begin{equation} \label{eq-Pi_i}
\Pi_i(v^a, f_{-i})(v^l)(s_l) = \max_{\phi_i \in F_i} \int_{X_{-i}} \int_{X_i} \Phi_i(s_l, x_i, x_{-i}, v^a, v^l) \phi_i(\rmd x_{i}|s_l) f_{-i}(\rmd x_{-i}|s_l).
\end{equation}

Let $\beta = \max\{\beta_i \colon i\in I\}$.
Then for any $v^a \in V_a$,  $v^l, \bar{v}^l \in V_l$, $x\in X$ and $s\in S_l$,
\begin{align*}
& \qquad \big| \Phi_i(s, x, v^a, v^l)  - \Phi_i(s,x,v^a,\bar{v}^l)   \big|  \leq \beta_i \sum_{s_l \in S_l} \big| v^l_i(s_l) - \bar{v}^l_i(s_l) \big| q(s_l|s,x)\lambda(s_l) \\
& \leq \beta_i \sup_{s_l\in S_l} \big| v^l_i(s_l) - \bar{v}^l_i(s_l) \big| \leq \beta \sup_{s_l \in S_l} \big| v^l_i(s_l) - \bar{v}^l_i(s_l) \big|.
\end{align*}
Thus, $\Pi$ is a $\beta$-contraction mapping. There is a unique $\bar v^{l} \in V_l$ such that $\Pi_i( v^a, f_{-i})(\bar v^{l})(s_l)  = \bar v_i^{l}(s_l)$ for each $i\in I$ and $s_l\in S_l$. Let $W(v^a, f)$ be the set of all $\phi \in F$ such that for each $i\in I$ and $s_l\in S_l$,
\begin{equation} \label{eq-bar-v_i^2}
\bar v_i^{l}(s_l) = \int_{X_{-i}} \int_{X_i} \Phi_i(s_l, x_i, x_{-i}, v^{a}, \bar v^{l}) \phi_i(\rmd x_{i}|s_l) f_{-i}(\rmd x_{-i}|s_l).
\end{equation}

Let $\bar v^{l}$ be the function on $S_l$ generated  by $v^a$ and $f$ as above, and $R(v^a,f)$ the set of $\lambda$-equivalence classes of $\cS$-measurable selections of $P(v^{a}, \bar v^l, \cdot)$ restricted to $S_a$. Then, the convex hull $\mbox{co}(R(v^a,f))$ is the set of $\lambda$-equivalence classes of $\cS$-measurable selections of $\mbox{co}\big(P(v^{a}, \bar v^l, \cdot)\big)$ restricted to $S_a$. Denote $\Psi(v^a, f) = \mbox{co}(R(v^a,f)) \times W(v^a, f)$ for each $v^a \in V_a$ and $f\in F$.

As shown in \cite{Nowak2003}, $\Psi$ is nonempty, convex, compact valued and upper hemicontinuous. By  Fan-Glicksberg's Fixed Point Theorem, $\Psi$ has a fixed point $(v^{a'}, f^{l'}) \in V_a\times F$.  Let $v^{l'}$ be the mapping from $S_l$ to $\bR^m$ that is generated  by $v^{a'}$ and $f^{l'}$ through the $\beta$-contraction mapping $\Pi$ as above. Then $v^{a'}$ is an $\cS$-measurable selection of $\mbox{co}\big(P(v^{a'}, v^{l'}, \cdot)\big)$ restricted to $S_a$; and furthermore
we have for each $i\in I$ and $s_l\in S_l$,
\begin{equation} \label{eq-bar-v_i^{l'}}
v_i^{l'}(s_l) = \int_{X_{-i}} \int_{X_i} \Phi_i(s_l, x_i, x_{-i}, v^{a'},  v^{l'}) f^{l'}_i(\rmd x_{i}|s_l) f^{l'}_{-i}(\rmd x_{-i}|s_l),
\end{equation}
\begin{equation} \label{eq-new-Pi_i}
\Pi_i(v^{a'}, f^{l'}_{-i})(v^{l'})(s_l)  =  v_i^{l'}(s_l).
\end{equation}

Following the same argument as in the proof of Theorem~\ref{thm-D_existence}, there exists an $\cS$-measurable selection $v^{a*}$ of $P_{(v^{a'}, v^{l'})}$ such that  $\mathbb{E}(v^{a*}\rho_j|\cG) = \mathbb{E}(v^{a'}\rho_j|\cG)$ for each $1\leq j \leq J$, where  the conditional expectation is taken on $(S_a, \cS^{S_a}, \lambda^{S_a})$ with $\lambda^{S_a}$ the normalized probability measure on $(S_a, \cS^{S_a})$. For any $s \in S$ and $x\in A(s)$, $\Phi_i(s,x,v^{a'}, v^{l'}) = \Phi_i(s,x,v^{a*}, v^{l'})$,  $\Gamma(v^{a'}, v^{l'}, s) = \Gamma(v^{a*}, v^{l'}, s)$, and therefore $P(v^{a'}, v^{l'}, s) = P(v^{a*}, v^{l'}, s)$. Thus,  $v^{a*}$ is an $\cS$-measurable selection of $P_{(v^{a*}, v^{l'})}$, and
 there exists an $\cS$-measurable  mapping  $f^{a*} \colon S_a \to \bigotimes_{i\in I}\cM(X_i)$  such that $f^{a*}(s)$ is a mixed-strategy Nash  equilibrium of the game $\Gamma(v^{a*}, v^{l'},s)$ and $v^{a*}(s)$ the corresponding equilibrium payoff.

Let $v^*(s)$ be $v^{a*}(s)$ for $s \in S_a$ and $ v^{l'}(s)$ for $s \in S_l$, and $f^*(s)$ be $f^{a*}(s)$ for $s \in S_a$ and $ f^{l'}(s)$ for $s \in S_l$. For $s_a \in S_a$, since $v^{a*}$ is a measurable selection of $P_{(v^{a*}, v^{l'})}$ on $S_a$, the equilibrium property of $f^{a*}(s_a)$ then implies that Equations~(\ref{eq-v_i}) and (\ref{eq-sge}) hold for $v^*$ and $f^*$. Next, for $s_l \in S_l$,  the identity $\Phi_i(s_l,x,v^{a'}, v^{l'}) = \Phi_i(s_l,x,v^{a*}, v^{l'})$ implies that Equations~(\ref{eq-bar-v_i^{l'}}) and (\ref{eq-new-Pi_i}) still hold  when $v^{a'}$ is replaced by $v^{a*}$, which means that Equations~(\ref{eq-v_i}) and (\ref{eq-sge}) hold for $v^*$ and $f^*$. Therefore, $f^*$ is a stationary Markov perfect equilibrium.

\subsection{Proof of Proposition~\ref{prop-noisy}}

In this subsection, we shall follow the notations in Subsection~\ref{subsec-shock}. As discussed in Remark~\ref{rmk-sges}, a noisy stochastic game is a stochastic game with endogenous shocks in which $H^d$ and $X^d$ are both singletons. As a result, we can slightly abuse the notations by viewing $Q_R$ to be a mapping from $H \times \cR$ to $[0, 1]$, and its Radon-Nikodym derivative $\psi$ to be a mapping defined on $H \times R$. For simplicity, we denote $\nu_h(\cdot) = Q_R(\cdot|h)$.

Let $\lambda(Z) = \int_H \int_R  \mathbf{1}_Z(h, r) \psi(r|h) \nu(\rmd r) \kappa(\rmd h)$ for all $Z \in  S$, and $\cG = \cH\otimes \{\emptyset, R\}$. Recall that $\phi(\cdot |s,x)$ is the Radon-Nikodym derivative of $Q_H(\cdot|s,x)$ with respect to $\kappa$. For each $(s,x)$,  $\phi(\cdot |s,x)$ is a mapping which does not depend on $r$, and hence is $\cG$-measurable. We need to show that $\cS$ has no $\cG$-atom under $\lambda$.

Fix any Borel subset $D\subseteq S$ with $ \lambda(D)  > 0$. There is a measurable mapping $\alpha$ from $(D,\cS^D)$ to $(L,\cB)$ such that $\alpha$ can generate the $\sigma$-algebra $\cS^D$, where $L$ is the unit interval endowed with the Borel $\sigma$-algebra $\cB$. Let $g(h,r) = h$ for each $(h,r) \in D$,  $D_h = \{r \colon (h,r) \in D \}$ and $H_D = \{h\in H: \nu_h (D_h) > 0\}$.

Denote $g_h(\cdot)=g(h,\cdot)$ and $\alpha_h(\cdot) = \alpha(h,\cdot)$ for each $h \in H_D$.  Define a mapping $f\colon H_D \times L \to [0,1]$ as follow: $f(h,l) = \frac{\nu_h \big( \alpha_h^{-1}([0,l])\big)}{\nu_h(D_h)}$.
Similarly, denote $f_h(\cdot)=f(h,\cdot)$ for each $h\in H_D$. For $\kappa$-almost all $h\in H_D$, the atomlessness of $\nu_h$ implies  $\nu_h\circ\alpha_h^{-1}(\{l\})=0$ for all $l\in L$. Thus, the distribution function $f_h(\cdot)$ is continuous on $L$ for $\kappa$-almost all $h \in H_D$.

Let $\gamma(s)=f(g(s), \alpha(s))$ for each $s \in D$, and  $D_0 =\gamma^{-1}([0,\frac{1}{2}])$, which is a subset  of $D$.
For $h\in H_D$, let $l_h$ be $\max\{l\in L\colon f_h(l) \le \frac{1}{2}\}$ if $f_h$ is continuous, and $0$ otherwise. When $f_h$ is continuous, $f_h(l_h) = 1/2$.
For any $E\in \cH$, let $D_1 = (E\times R)\cap D$, and $E_1 = E\cap H_D$.
If $\lambda(D_1) = 0$, then
\begin{align*}
& \qquad \lambda(D_0\setminus D_1) = \lambda(D_0) = \int_{H_D} \nu_h \circ \alpha_h^{-1}\circ f_h^{-1} \big( [0,\frac{1}{2}] \big) \kappa(\rmd h) \\
& = \int_{H_D} \nu_h \big( \alpha_h^{-1}([0,l_h])\big) \kappa(\rmd h) = \int_{H_D} f(h,l_h) \nu_h(D_h) \kappa(\rmd h)  \\
& = \frac{1}{2} \int_{H_D} \nu_h(D_h) \kappa(\rmd h) = \frac{1}{2}\lambda(D)  > 0.
\end{align*}
If $\lambda(D_1) > 0$, then
\begin{align*}
& \quad \lambda(D_1\setminus D_0) = \int_{E_1} \int_R  \mathbf{1}_{D\setminus D_0}(h, r) \nu_h(\rmd r) \kappa(\rmd h) = \int_{E_1} \nu_h \circ \alpha_h^{-1}\circ f_h^{-1} \big( (\frac{1}{2},1] \big) \kappa(\rmd h) \\
& =  \int_{E_1} \nu_h \circ \alpha_h^{-1}\circ f_h^{-1} \big( [0, 1] \setminus [0,\frac{1}{2}] \big) \kappa(\rmd h) = \frac{1}{2} \int_{E_1} \nu_h(D_h) \kappa(\rmd h) = \frac{1}{2} \lambda(D_1)  > 0.
\end{align*}
Hence, $D$ is not a $\cG$-atom. Therefore, $\cS$ has no $\cG$-atom and the condition of coarser transition kernel is satisfied.

\subsection{Proof of Proposition~\ref{prop-levy}}\label{sub-levy}

Suppose that the atomless part of the state transition satisfies the decomposable coarser transition kernel condition with respect to some probability measure $\lambda$ on $(S, \cS)$. Given the form of the state transition $Q(s, x)$, the atomless part of $\lambda$ concentrates on $S_a = [0, 1)$ while $\lambda$ has an atom at $S_l = \{1\}$. For simplicity, we replace the notation $(S_a,\cS^{S_a},\lambda|_{S_a})$ as used in Section \ref{sec-atomic}.
Let $\tilde{S}$ be $S_a = [0, 1)$, $\tilde{\lambda}$ the restriction $\lambda|_{\tilde{S}}$, and $\tilde{\cS}$ the Borel $\sigma$-algebra on $\tilde{S} = [0, 1)$. For some positive integer $J$, the Radon-Nikodym derivative $\tilde q(\tilde{s}|s, x)$ ($\tilde{s} \in \tilde{S}$) of (the atomless part)  $\tilde{Q}(s, x)U(s,1)$ with respect to $\lambda$ (and $\tilde{\lambda}$ as well) can be expressed as $\sum_{1\le j \le J} \tilde{q}_j(\tilde{s},s, x) \rho_j(\tilde{s})$ for any $s\in S$ and $x \in X$, where $\tilde q_j$ is product measurable, $\tilde{q}_j(\cdot, s, x)$ and $\rho_j (\cdot)$ are nonnegative and $\tilde{\lambda}$-integrable. Moreover, for some sub-$\sigma$-algebra $\cG$ of $\tilde{\cS}$ on $\tilde{S} =[0, 1)$, $\tilde q_j(\cdot,s, x)$ is $\cG$-measurable for each $j, s, x$, and $\tilde{\cS}$ has no $\cG$-atom under $\tilde{\lambda}$.

Fix any $s \in [0, 1)$, and an action profile $x^0$ with players $C$, $C'$, $D$ and $D'$ playing the strategy $1$; then $\tilde{Q}(s, x^0) = \frac{(1 - s)}{16}$. It follows form the above paragraph that
$U(s,1)$ is absolutely continuous with respect to $\tilde{\lambda}$ with the corresponding Radon-Nikodym derivative $q(\tilde{s}|s)$ ($\tilde{s} \in [0, 1)$) to be $\sum_{1\le j \le J} \frac{16}{(1 - s)} \tilde{q}_j(\tilde{s},s, x^0) \rho_j(\tilde{s})$. Denote $\frac{16}{(1 - s)} \tilde{q}_j(\tilde{s},s, x^0)$ by $q_j (\tilde{s}, s)$. Then, we have $q(\tilde{s}|s) = \sum_{1\le j \le J} q_j(\tilde{s},s) \rho_j(\tilde{s})$ for any $\tilde{s}\in [0,1)$, where $q_j (\cdot, s)$ is $\cG$-measurable.

It follows from the previous paragraph that the Lebesgue measure $\eta = U(0,1)$ is absolutely continuous with respect to $\tilde{\lambda}$ with the corresponding Radon-Nikodym derivative $q(\cdot|0)$ (to be denoted by  $\bar{q}(\cdot)$ for simplicity).
Let $\tilde{D} = \{\tilde{s} \in [0,1) \colon \bar{q}(\tilde{s}) > 0\}$. Then $\tilde{\lambda}(\tilde{D}) > 0$ and $\eta(\tilde{D}^c) = \int_{\tilde{D}^c} \bar{q}(\tilde{s}) \tilde{\lambda}(\rmd \tilde{s}) = 0$, where $\tilde{D}^c$ is the complement of the set $\tilde{D}$ in $[0, 1)$.
The Radon-Nikodym derivative of $U(s,1)$  with respect to the Lebesgue measure $\eta$ on $[0, 1)$ is:
$$\hat{q}(\tilde{s}|s) =
\begin{cases}
\frac{1}{1-s} & \tilde{s}\in [s,1), \\
0             & \tilde{s}\in [0,s).
\end{cases}
$$
Hence, the Radon-Nikodym derivative $q(\cdot|s)$ of $U(s,1)$ with respect to $\tilde{\lambda}$ can be expressed as $\hat{q}(\cdot|s) \bar{q}(\cdot)$. Thus, we have
$q(\tilde{s}|s)= \hat{q}(\tilde{s}|s) \bar{q}(\tilde{s})= \sum_{1\le j \le J} q_j(\tilde{s},s) \rho_j(\tilde{s})$ for any $\tilde{s}\in [0,1)$.

Denote $D_j = \{\tilde{s}\in \tilde{D} \colon \rho_j(\tilde{s}) = 0\}$ for  $1\le j \le J$. Since $\bar{q}(\tilde{s}) > 0$ for all $\tilde{s}\in \tilde{D}$ and $\hat{q}(\tilde{s}|s) > 0$ for $0\le s \le \tilde{s} < 1$, we must have $\cap_{1\le j\le J} D_j = \emptyset$,  and hence $\tilde{\lambda} \big( \cap_{1\le j\le J} D_j \big) = 0$.

First suppose that $\tilde{\lambda}(D_j) = 0$ for all $j$. Let $\bar D =  \cup_{1\le j\le J} D_j$; then $\tilde{\lambda}(\bar D) =0$. Fix $s'\in [0,1)$. Let $E_j = \{\tilde{s}\in \tilde{D}\colon q_j(\tilde{s},s') = 0\}$ and $E_0 = \cap_{1\le j\le J} E_j$. Then $E_j \in \cG^{\tilde{D}}$ for $1\le j \le J$, and hence $E_0 \in \cG^{\tilde{D}}$. For any $\tilde{s}\in [s',1) \cap \tilde{D}$, since $q(\tilde{s}|s') = \hat{q}(\tilde{s}|s') \bar{q}(\tilde{s})> 0$, there exists  $1\le j \le J$ such that $q_j(\tilde{s}|s') > 0$, which means that $\tilde{s} \notin E_j$ and $\tilde{s} \notin E_0$. Hence, $E_0 \subseteq [0, s') \cap \tilde{D}$.  For any $\tilde{s} \in \left( ([0,s') \cap \tilde{D}) \setminus \bar D \right)$, we have $q(\tilde{s}|s') = 0$ and $\rho_j(\tilde{s}) > 0$ for each $1\le j \le J$, which implies that $q_j(\tilde{s}|s') = 0$ for each $1\le j \le J$, and $\tilde{s} \in E_0$. That is, $\left( ([0,s') \cap \tilde{D}) \setminus \bar D \right) \subseteq E_0$. Hence, $\tilde{\lambda}(E_0 \triangle ([0,s')\cap \tilde{D})) =0$. Therefore, $[0,s')\cap \tilde{D} \in \cG^{\tilde{D}}$ for all $s'\in [0,1)$. Since the class of intervals $\{[0,s')\}_{s'\in [0,1)}$ generates the Borel $\sigma$-algebra on $[0, 1)$, we obtain that $\cG^{\tilde{D}}$ coincides with ${\tilde{\cS}}^{\tilde{D}}$ under $\tilde{\lambda}$. Thus,  $\tilde{\cS}$ has a $\cG$-atom $\tilde{D}$ under $\tilde{\lambda}$. This is a contradiction.

Next suppose that $\tilde{\lambda}(D_j) = 0$ does not hold for all $j$. Then there exists a set, say $D_1$, such that $\tilde{\lambda}(D_1) > 0$. Let $Z = \{K\subseteq \{1,\ldots, J\}\colon 1\in K, \tilde{\lambda}(D^K) > 0\}$, where $D^K = \cap_{j \in K} D_j$. Hence, $\{1\} \in Z$,   $Z$ is finite and nonempty. Let $K_0$ be the element in $Z$ containing most integers; that is, $\big|K_0 \big| \ge \big| K\big|$  for any $K\in Z$, where $\big| K\big|$ is the cardinality of $K$. By the definition of $K_0$, $\tilde{\lambda} \left( D^{K_0} \right) >0$.
Let $K_0^c =  \{1,\ldots, J\}\setminus K_0$.  Then $K_0^c$ is nonempty since $\tilde{\lambda}\big( \cap_{1\le j\le J} D_j \big) = 0$. In addition, $\tilde{\lambda}\big( D^{K_0} \cap D_j \big) = 0$ for any $j \in K_0^c$. Otherwise, $\tilde{\lambda}\big( D^{K_0} \cap D_j \big) > 0$ for some $j \in K_0^c$ and hence the set $K_0\cup \{j\}$ is in $Z$, which contradicts the choice of $K_0$.   Let $\hat D = \cup_{k\in K_0^c} \big( D^{K_0} \cap D_k \big)$; then $\tilde{\lambda}(\hat D) =0$.  For all $\tilde{s}\in D^{K_0}$, $q(\tilde{s}|s) = \sum_{k\in K_0^c}q_k(\tilde{s},s) \rho_k(\tilde{s})$ for all $s\in [0,1)$.

Fix $s'\in [0,1)$. Let $E_k = \{\tilde{s}\in \tilde{D}: q_k(\tilde{s},s') = 0\}$ and $E_{K_0^c} = \cap_{k\in K_0^c} E_k$. Then $E_k \in \cG^{\tilde{D}}$ for any $k$ and hence $E_{K_0^c} \in \cG^{\tilde{D}}$.  For any $\tilde{s}\in [s',1) \cap \tilde{D}$, since $q(\tilde{s}|s') > 0$, there exists $k \in {K_0^c}$ such that $q_k(\tilde{s}, s') > 0$, which means that $\tilde{s} \notin E_k$ and $\tilde{s} \notin E_{K_0^c}$. Hence, $E_{K_0^c} \subseteq [0, s') \cap \tilde{D}$, and $E_{K_0^c} \cap D^{K_0} \subseteq [0, s') \cap D^{K_0}$.  Now, for any $\tilde{s} \in \left( \left( [0,s') \cap D^{K_0} \right) \setminus \hat D \right)$, we have $q(\tilde{s}|s') = 0$, and $\rho_k(\tilde{s}) > 0$ for each $k\in K_0^c$, which implies that $q_k(\tilde{s},~s') = 0$ for each $k\in K_0^c$, and $\tilde{s} \in E_{K_0^c}$. That is, $\left( \left( [0,s') \cap D^{K_0} \right) \setminus \hat D \right) \subseteq E_{K_0^c} \cap D^{K_0}$. Hence, $([0,s']\cap D^{K_0}) \setminus (E_{K_0^c}\cap D^{K_0}) \subseteq \hat D$, and $\tilde{\lambda}\big( (E_{K_0^c}\cap D^{K_0}) \triangle ([0,s']\cap D^{K_0}) \big) =0$.
Therefore, $[0,s']\cap D^{K_0} \in \cG^{D^{K_0}}$ for all $s'\in [0,1)$. By the fact that the class of intervals $\{[0,s')\}_{s'\in [0,1)}$ generates the Borel $\sigma$-algebra on $[0, 1)$, we obtain that $\cG^{D^{K_0}}$ coincides with ${\tilde{\cS}}^{D^{K_0}}$ under $\tilde{\lambda}$. Thus, $\tilde{\cS}$ has a $\cG$-atom $D^{K_0}$ under $\tilde{\lambda}$.
This is again a contradiction.

\subsection{Proof of Propositions~\ref{prop-atom} and \ref{prop-finite}}

\begin{proof}[Proof of Proposition~\ref{prop-atom}]
Define a correspondence
$$G(s) =
\begin{cases}
\{0,1\} & s\in D;\\
\{0\}   & s\notin D.
\end{cases}$$
We claim that $\cI^{(\cS,\cG,\lambda)}_G \neq \cI^{(\cS,\cG,\lambda)}_{\mbox{co} (G)}$.
Let $g_1(s) = \frac{1}{2} \mathbf{1}_{D}$, where $\mathbf{1}_D$ is the indicator function of the set $D$. Then $g_1$ is an $\cS$-measurable selection of $\mbox{co}(G)$. If there is an $\cS$-measurable selection $g_2$ of $G$ such that $\bE^{\lambda}(g_1|\cG) = \bE^{\lambda}(g_2|\cG)$, then there is a subset $D_2 \subseteq D$ such that $g_2(s) = \mathbf{1}_{D_2}$.
Since $D$ is a $\cG$-atom, for any $\cS$-measurable subset $E\subseteq D$, there is a subset $E_1\in \cG$ such that $\lambda(E \triangle( E_1\cap D)) = 0$.  Then
\begin{align*}
\lambda(E\cap D_2)
& = \int_S \mathbf{1}_E(s) g_2(s) \lambda(\rmd s)  = \int_S \bE^{\lambda} \big( \mathbf{1}_{E_1} \mathbf{1}_D g_2|\cG \big) \rmd\lambda = \int_S \mathbf{1}_{E_1} \bE^{\lambda} \big(g_2|\cG \big) \rmd\lambda     \\
& = \int_S  \mathbf{1}_{E_1} \bE^{\lambda} \big(g_1|\cG \big) \rmd\lambda = \frac{1}{2} \int_S  \mathbf{1}_{E_1} \mathbf{1}_{D} \rmd\lambda = \frac{1}{2}\lambda(E_1 \cap D) = \frac{1}{2}\lambda(E).
\end{align*}
Thus, $\lambda(D_2) = \frac{1}{2}\lambda(D) > 0$ by choosing $E = D$. However, $\lambda(D_2) =  \frac{1}{2}\lambda(D_2)$ by choosing $E = D_2$, which implies that $\lambda(D_2) = 0$. This is a contradiction.
\end{proof}

\begin{proof}[Proof of Proposition~\ref{prop-finite}]
Suppose that there exists an $\cS$-measurable selection $g$ of $G$ such that $\bE^{\lambda}(g \rho_n|\cF) = 0$ for any $n\ge 0$. Then there exists a set $E\in \cS$ such that
$$g(s) = \begin{cases}
1  & s\in E;\\
-1 & s\notin E.
\end{cases}$$
Thus,
$$\lambda(E) -\lambda(E^c) = \int_S g\rho_0 \rmd\lambda =\int_S \bE^{\lambda}(g\rho_0|\cF) \rmd\lambda =  0,$$ which implies $\lambda(E) = \frac{1}{2}$. Moreover,
$$\int_S g\varrho_n \rmd\lambda = \int_S g\rho_n \rmd\lambda - \int_S g \rmd\lambda = \int_S \bE^{\lambda}(g\rho_n|\cF)  \rmd\lambda - 0 = 0$$
for each $n\ge 1$, which contradicts with the condition that $\{\varrho_n\}_{n\ge 0}$ is a complete orthogonal system.
\end{proof}

{\small
\singlespacing

\end{document}